\documentclass{article}
\usepackage{amsfonts,amsmath,amsthm,hyperref}
\usepackage{xcolor}
\usepackage{tikz}
\newcounter{satznum}
\newtheorem{theorem}{Theorem}[satznum]

\numberwithin{theorem}{subsection} 

\newtheorem{lemma}[theorem]{Lemma}

\newtheorem{proposition}[theorem]{Proposition}
\newenvironment{acknowledgement}
 {\begin{trivlist}\item[]{\bf Acknowledgements.}}
 {\end{trivlist}}
\newenvironment{remark}
 {\begin{trivlist}\item[]{\bf Remark}}
 {\end{trivlist}}

\newenvironment{example}
 {\begin{trivlist}\item[]{\bf Example.}}
 {\end{trivlist}}

{\makeatletter
\gdef\me{{\mathbb E}} 
\gdef\nz{{\mathbb N}} 
\gdef\pr{{\mathbb P}} 
\gdef\rz{{\mathbb R}} 
}
\newcounter{todocounter}

\makeatletter
\def\@MRExtract#1 #2!{#1}
\newcommand{\MR}[1]{
  \xdef\@MRSTRIP{\@MRExtract#1 !}
  \href{http://www.ams.org/mathscinet-getitem?mr=\@MRSTRIP}{MR\@MRSTRIP}}
\makeatother
\newcommand{\scaling}{
\begin{tikzpicture}[->, thick, node distance=3cm, auto]
  \node (Xn) at (0,0) {$\frac{1}{N}X^{(N)}$};
  \node (Yn) at (0,-3) {$Y^{(N)}$};
  \node (X) at (3.5,0) {$X$};
  \node (Y) at (3.5,-3) {$Y$};

  \node at (-2,0.8) {discrete};
  \node at (-2,0.3) {space};
  \node at (-2,-0.2) {$\{0, \tfrac{1}{N},\tfrac{2}{N},\dots,1\}$};

  \node at (-2,-2.8) {discrete};
  \node at (-2,-3.3) {space};
  \node at (-2,-3.8) {$\{0,1,\dots,N\}$};

  \node at (5,0.8) {continuous};
  \node at (5,0.3) {space};
  \node at (5,-0.2) {$[0,1]$};

  \node at (5,-2.8) {discrete};
  \node at (5,-3.3) {space};
  \node at (5,-3.8) {$\nz_0$};

  \draw[->] (Xn) -- (X) node[midway,above] {$N\to\infty$};
  \draw[->] (Yn) -- (Y) node[midway,above] {$N\to\infty$};

  \draw[<->] (Xn) -- (Yn) node[midway,left] {$D^{(N)}$};
  \draw[<->] (X) -- (Y) node[midway,right] {$D$};
\end{tikzpicture}
}

\newcommand{\betapicture}{
\begin{tikzpicture}[scale=1.2]
  \draw[->] (0,0) -- (5,0) node[right] {$a$};
  \draw[->] (0,0) -- (0,4) node[above] {$b$};

  \draw[very thick] (1,0) -- (1,4) node[above] {$a=1$};
  \draw[thick] (2,0) -- (2,4) node[above] {$a=2$};

  \draw[thick] (0,1) -- (1,0) node[below] {};
  \draw (0,0) -- (4,4) node[right] {$a=b$};

  \node at (-.8,3.) {$\tfrac{1}{2}\delta_0 + \tfrac{1}{2}\delta_1$};
  \node at (1.45,3.3) {$\beta(b,b)$};
  \node at (3.8,3.3) {$\beta(1,1)$};
  \node at (2.5,1.3) {$\beta(\frac{b}{2},\frac{b}{2})$};

  \node at (3.8,+0.3) {$\delta_{\frac{1}{2}}$};

  \draw (1,-0.) -- (1,0.) node[below=3pt] {1};
  \draw (2,-0.) -- (2,0.) node[below=3pt] {2};

  \node[rotate=90] at (0.85,2) {HARMONIC};
\end{tikzpicture}
}

\title

\begin{document}
   \section*{On a $\Lambda$-mutation model and the harmonic model}
   \textsc{Cristian Giardin\`a}\footnote{Department of Mathematics, University of Modena and Reggio Emilia, Via G. Campi 213/b, 44125 Modena, Italy, E-mail address: cristian.giardina@unimore.it} and
   \textsc{Martin M\"ohle}\footnote{Department of Mathematics, University of T\"ubingen, Auf der Morgenstelle 10, 72076 T\"ubingen, Germany, E-mail address: martin.moehle@uni-tuebingen.de}

   \begin{center}
      \today
   \end{center}

\begin{abstract}
   We introduce a continuous-time mutation model with two types determined by a finite measure $\Lambda$ on the unit interval. The model satisfies a certain consistency property known from mathematical population genetics and includes so called harmonic models being of interest in mathematical statistical physics. We mainly focus on the situation when the number of particles is equal to some constant $N$. Duality results and scaling limits as $N\to\infty$ for the forward and backward processes are provided leading to a commutative diagram. The stationary distribution of the forward process is studied with an emphasis on the case when $\Lambda$ is a beta distribution. The work bridges particle models from mathematical statistical physics and mutation models from mathematical population genetics.

   \vspace{2mm}
   
   \noindent
   Keywords: Consistency; duality; harmonic model; $\Lambda$-mutation model; stationary distribution
\end{abstract}

\subsection{Introduction}

In this section we define the class of models that will be studied in this paper, we provide their interpretation and we explain how the paper is structured.

\subsubsection{Model definition}
Let $N\in\nz:=\{1,2,\ldots\}$. We study a class of continuous-time Markov processes $X^{(N)}=(X_t^{(N)})_{t\ge 0}$ with state space $E_N:=\{0,\ldots,N\}$ and generator $L^{X^{(N)}}$ acting on functions $f:E_N\to\rz$ via
\begin{equation} \label{onedimgen}
   L^{X^{(N)}}f(n)\ =\ \sum_{k=1}^n \varphi(k,n)\big(f(n-k)-f(n)\big)
   + \sum_{k=1}^{N-n}\varphi(k,N-n)\big(f(n+k)-f(n)\big),\quad n\in E_N,
\end{equation}
where empty sums are defined as zero. The nonnegative rates $\varphi(k,n)$, $1\le k\le n$, are assumed to satisfy the consistency relations
\begin{equation} \label{consis}
   \varphi(k,n)\ =\ \bigg(1-\frac{k}{n+1}\bigg)\varphi(k,n+1) + \frac{k+1}{n+1}\varphi(k+1,n+1),\quad n\in\nz,k\in\{1,\ldots,n\}.
\end{equation}
The consistency relations (\ref{consis}) are equivalent (see Proposition \ref{equivalent}) to the commutator equations
\begin{equation} \label{commutator1}
   [L^{X^{(N)}},A^{(N)}]\ =\ 0,\qquad N\in\nz,
\end{equation}
where $A^{(N)}$ denotes the annihilation operator acting as $A^{(N)}f(n)=nf(n-1)+(N-n)f(n)$, $n\in E_N$. Thus, for a consistent system, the action of removing at random a particle commutes with the evolution.

Proposition \ref{equivalent} also shows that the consistency relations (\ref{consis}) are equivalent to the existence of a finite measure $\Lambda$ on the unit interval $[0,1]$ such that
\begin{equation} \label{lambda_rep}
   \varphi(k,n)\ =\ \binom{n}{k}\int_{[0,1]} t^{k-1}(1-t)^{n-k}\,\Lambda({\rm d}t),\qquad n\in\nz, k\in\{1,\ldots,n\}.
\end{equation}
For this reason, we call the process $X^{(N)}$ with generator \eqref{onedimgen} the `{\em $\Lambda$-mutation model}'.

\subsubsection{Model interpretation}
The model introduced above can be seen as a mutation model in mathematical population genetics. To see this it is convenient to define $\nz_0:=\{0,1,2,\ldots\}$ and remark that the one-dimensional process $X^{(N)}$ is obtained from a two-dimensional (two-type) process with state space $\nz_0^2$, whose generator acts on bounded test functions $f:\nz_0^2\to\rz$ via
\begin{eqnarray} \label{twodimgen}
   \hspace{-5mm}
      Lf(n_1,n_2)
      & = & \sum_{k=1}^{n_1}\varphi(k,n_1)\big(f(n_1-k,n_2+k)-f(n_1,n_2)\big)\nonumber\\
      &   & +
            \sum_{k=1}^{n_2}\varphi(k,n_2)\big(f(n_1+k,n_2-k)-f(n_1,n_2)\big),
            \qquad n_1,n_2\in\nz_0.
\end{eqnarray}
The generator (\ref{onedimgen}) is obtained from (\ref{twodimgen}) when restricting on $n_1+n_2=N$. Again, the consistency relations (\ref{consis}) are equivalent (see Proposition \ref{equivalent}) to the commutator relation
\begin{equation} \label{commutator2}
   [L,A]\ =\ 0,
\end{equation}
where $A$ denotes the annihilation operator acting as $Af(n_1,n_2)=n_1f(n_1-1,n_2)+n_2f(n_1,n_2-1)$.
The two-dimensional process might be interpreted in terms of a population genetics model as follows. Consider a population, where each individual has one of two possible types. If the process is in state $(n_1,n_2)$, then $k\in\{1,\ldots,n_1\}$ of the $n_1$ particles of type $1$ mutate to type $2$ at the rate $\varphi(k,n_1)$
and $k\in\{1,\ldots,n_2\}$ of the $n_2$ particles of type $2$ mutate to type $1$ at the rate $\varphi(k,n_2)$.
The study of this model is partly motivated from related forward and backward processes known from mathematical population genetics models, namely the $\Lambda$-Fleming--Viot process (see, for example, Griffiths \cite[Eq.~(5)]{Griffiths2014}) and its dual counter part, the block counting process of the $\Lambda$-coalescent. Note however that the latter process has state space $\nz$ and absorbing state $1$ (corresponding to the most recent common ancestor), whereas the dual process $Y$ of our $\Lambda$-mutation model (see Theorem \ref{main2}) has state space $\nz_0$ and absorbing state $0$. Although there does not seem to be a direct relation between our mutation model and those models from mathematical population models, they share the common structure of duality linking the forward and backward evolution.

Another interpretation of the $\Lambda$-mutation model is in the framework of mathematical statistical physics. Indeed, the process with generator \eqref{twodimgen} belongs to the class of stochastic exchange models, which are a family of Markov processes in which a quantity, such as energy or wealth or mass, is randomly exchanged between agents interacting through the edges of a graph. The stochastic exchange models include, among others, the {\em averaging process} \cite{AldousLanoue2012} and the {\em Kipnis--Marchioro--Presutti model} \cite{KipnisMarchioroPresutti1982}. If we interpret $(n_1,n_2)$ as the number of particles at two sites of a graph, then we see from \eqref{twodimgen} that in the $\Lambda$-mutation model the redistribution rule is a function of the number of particles at the departure site only. In particular, if the measure $\Lambda$ is the uniform distribution on $[0,1]$, then the rate $\varphi(k,n)=1/k$ does not depend on $n$ and the model reduces to the so called {\em harmonic model} studied in the mathematical statistical physics literature on integrable particle systems (see, for example, \cite{FrassekGiardinaKurchan2020}). The specific choice of $\Lambda$ as the uniform measure (or more generally as the $\beta(1,b)$-distribution) yields a direct connection to integrability, in the sense of the Yang--Baxter equation \cite{FrassekGiardina2022}. This connection has been exploited to identify in closed form the stationary measure for boundary-driven one-dimensional chains  \cite{CarinciFranceschiniGabrielliGiardinaTsagkarogiannis2024,CarinciFranceschiniFrassekGiardinaRedig2023},
which is rarely available in statistical physics. For more general choices of the measure $\Lambda$ and the study of intertwining and dynamical properties (e.g. spectral gap) we refer the reader to \cite{GiardinaRedigvanTol2025} and \cite{KimQuattropaniSau2025}.

The mathematical analysis of both `worlds' has much in common and essentially all results presented in this paper strengthen the deep bridge between particle models from mathematical statistical physics and mutation models from mathematical population genetics.

\subsubsection{Main contributions}
The main contributions of this paper are as follows. We introduce a class of continuous-time mutation models whose transition rates satisfy a natural consistency property, and we show that this property is equivalent to the existence of a finite measure $\Lambda$ on the unit interval. This representation provides a unified framework encompassing both mutation models from population genetics and stochastic exchange models from statistical physics. We further analyze the behavior of these models in the large population limit and identify a limiting Markov process on $[0,1]$, whose generator has an explicit nonlocal form reflecting the underlying $\Lambda$-mechanism.
A key contribution of the paper is the systematic identification of duality relations. We construct dual processes for both the discrete and continuous models, showing that the forward dynamics can be related to suitable death processes through explicit duality functions. These dualities are derived from the algebraic structure of the generator, in particular from its action on monomials and factorial moments, and provide a powerful tool to analyze the evolution. In particular, they allow one to transfer information between the forward process and its dual, yielding insight into moment dynamics, scaling limits, and structural properties of the model. Finally, we investigate the stationary distribution of the limiting process. We obtain a recursive characterization of its moments and a functional equation for its moment generating function. In the important special case where $\Lambda$ is a beta distribution, we derive more explicit results, including a characterization of when the stationary distribution is itself a beta distribution, as well as regimes exhibiting reversibility and connections to integrable (harmonic) models.

\subsubsection{Paper organization}
This article is organized as follows. The main results are collected in Section \ref{results} with an emphasis on consistency in Subsection \ref{consistency} and scaling limits and duality in Subsection \ref{scalingandduality}, leading to the commutative diagram illustrated in Figure \ref{fig1}. A more advanced duality result is provided in Subsection \ref{advanced}. Moreover, Section \ref{results} contains detailed information on the model for the particular case when $\Lambda$ is a beta distribution (see Section \ref{beta}). Proofs are provided in Section \ref{proofs}.

\color{black}

\subsection{Results} \label{results}
\subsubsection{Consistency} \label{consistency}
We start our results with equivalent conditions relating consistency to commutator equations and to a measure $\Lambda$ on the unit interval.
\begin{proposition}[Consistency] \label{equivalent}
   Let $(\varphi(k,n))_{n\in\nz,k\in\{1,\ldots,n\}}$ be a triangular array of nonnegative real numbers. Then the following four statements are equivalent.
   \begin{enumerate}
      \item[(i)] The array $(\varphi(k,n))_{n\in\nz,k\in\{1,\ldots,n\}}$ satisfies the consistency relations (\ref{consis}).
      \item[(ii)] There exists a finite measure $\Lambda$ on the unit interval $[0,1]$ satisfying (\ref{lambda_rep}).
      \item[(iii)] The generator $L$, defined via (\ref{twodimgen}), satisfies the commutator relation (\ref{commutator2}).
      \item[(iv)] For every $N\in\nz$, the generator $L^{X^{(N)}}$, defined via (\ref{onedimgen}), satisfies the commutator relation (\ref{commutator1}).
   \end{enumerate}
\end{proposition}

\subsubsection{Scaling limits and duality} \label{scalingandduality}
Our first main result is a scaling limit for the process $X^{(N)}$ as $N\to\infty$.
\begin{theorem}[Scaling limit for $X^{(N)}/N$] 
   \label{main1}
   Assume that $X_0^{(N)}/N\to X_0$ in distribution as $N\to\infty$ for some random variable $X_0$. Then, as $N\to\infty$, the space-scaled process $(X_t^{(N)}/N)_{t\ge 0}$ converges in $D_{[0,1]}([0,\infty))$ to the process $X=(X_t)_{t\ge 0}$ having state space $[0,1]$, whose generator $L^X$ acts on test functions $f\in C^2([0,1])$ via
   \begin{equation} \label{generator_general}
   L^Xf(x)
   \ =\ \Lambda(\{0\})(1-2x)f'(x)+\int_{(0,1]}
        \big(f(x(1-t))+f(x(1-t)+t)-2f(x)\big)
        \frac{\Lambda({\rm d}t)}{t},
   \end{equation}
   where $\Lambda$ is the finite measure on $[0,1]$ uniquely determined via (\ref{lambda_rep}).
\end{theorem}
\begin{remark} (Two-type process).
   As in the case of the process $X^{(N)}$ with discrete state space, also the process $X$ is obtained from a two-dimensional process. This is the process with state space $[0,\infty)^2$ whose generator acts on test functions $f\in C^2([0,\infty)^2)$ via
   \begin{eqnarray} \label{twodimgen-cont}
      {\cal L}f(x_1,x_2)
      & = & -\Lambda(\{0\})(x_1-x_2)\left(\frac{\partial}{\partial x_1}-\frac{\partial}{\partial x_2}\right)f(x_1,x_2)\nonumber\\
      &   & +\int_{(0,1]}\big(f(x_1-x_1 t,x_2+x_1 t )-f(x_1,x_2)\big)
            \frac{\Lambda({\rm d}t)}{t} \nonumber\\
      &   & +\int_{(0,1]}\big(f(x_1+x_2 t,x_2-x_2 t )-f(x_1,x_2)\big)
            \frac{\Lambda({\rm d}t)}{t}.
   \end{eqnarray}
   The generator \eqref{generator_general} is obtained from \eqref{twodimgen-cont} when restricting to $x_1+x_2=1$. Furthermore, when $\Lambda$ is the uniform measure on the interval $[0,1]$, because of the occurrence of the $1/t$-term in the integral in (\ref{generator_general}), we speak of the continuous harmonic model.
\end{remark}
The following lemma clarifies the action of the generator $L^X$ on monomials, a result that will be central to prove duality later.
\begin{lemma}[Action of $L^X$ on monomials] \label{generator_monomial}
   The generator $L^X$ acts on the $n$-th monomial $f_n(x):=x^n$ via
   \begin{equation} \label{monomial}
      L^Xf_n(x)\ =\ \sum_{k=1}^n \varphi(k,n)x^{n-k}-2\varphi(n)x^n,
      \qquad n\in\nz_0,
   \end{equation}
   where $\varphi(k,n)$ is defined via (\ref{lambda_rep}) and
   the total rate $\varphi(n)$ is defined via
   \begin{equation} \label{totalrate}
      \varphi(n)
      \ :=\ \sum_{k=1} ^n\varphi(k,n)
      \ =\ dn+\kappa+\int_{(0,1)}\frac{1-(1-t)^n}{t}\,\Lambda({\rm d}t),
      \qquad n\in\nz_0,
   \end{equation}
   with drift $d:=\Lambda(\{0\})$ and killing rate $\kappa:=\Lambda(\{1\})$. In particular, $L^X$ acts on the identity $f_1$ via $L^Xf_1(x)=(1-2x)\Lambda([0,1])$.
\end{lemma}
\begin{remark} (Cone property).
   For $n\in\nz_0$ let $P_n$ denote the space of all polynomials on $[0,1]$ of degree less than or equal to $n$. Eq.~(\ref{monomial}) in particular implies that $L^XP_n\subseteq P_n$ for all $n\in\nz_0$, a crucial and useful property of $L^X$.
\end{remark}
\begin{remark} (The rates are determined by the total rates).
   Let $n\in\nz$ and $k\in\{1,\ldots,n\}$ and assume for the moment that $\Lambda$ has no mass at $0$. Then, by (\ref{lambda_rep}),
   \begin{eqnarray*}
      \varphi(k,n)
      & = & \binom{n}{k}
            \int_{(0,1]}\big(1-(1-t)\big)^k(1-t)^{n-k}
            \,\frac{\Lambda({\rm d}t)}{t}\\
      & = & \binom{n}{k}
            \int_{(0,1]}\sum_{j=0}^k\binom{k}{j}\big(-(1-t)\big)^{k-j}(1-t)^{n-k}
            \,\frac{\Lambda({\rm d}t)}{t}\\
      & = & \binom{n}{k}\int_{(0,1]}
            \sum_{j=0}^k \binom{k}{j}(-1)^{k-j}(1-t)^{n-j}\,\frac{\Lambda({\rm d}t)}{t}.
   \end{eqnarray*}
   Exploiting the relation $\sum_{j=0}^k\binom{k}{j}(-1)^{k-j}=(1-1)^k=0$ yields
   \begin{eqnarray}
      \varphi(k,n)
      & = & \binom{n}{k}\int_{(0,1]}
            \sum_{j=0}^k \binom{k}{j}(-1)^{k-j}\big((1-t)^{n-j}-1\big)
            \frac{\Lambda({\rm d}t)}{t}\nonumber\\
      & = & \binom{n}{k}\sum_{j=0}^k \binom{k}{j}(-1)^{k-j+1}\varphi(n-j).
      \label{ratesrelation}
   \end{eqnarray}
   It is easily seen that (\ref{ratesrelation}) remains valid if $\Lambda(\{0\})>0$. The rates $\varphi(k,n)$ can hence be computed from the total rates (\ref{totalrate}) via (\ref{ratesrelation}). Similar relations between the rates and the total rates are known for coalescent processes with multiple collisions; see, for example, \cite[Eq.~(16)]{Moehle2006}.
\end{remark}
\begin{remark} (Relation to subordinators).
   The right-hand side in (\ref{totalrate}) is even defined with $n$ replaced by an arbitrary real number $\eta\ge 0$. Let $Q$ denote the measure on $(0,1)$ defined via $Q({\rm d}t):=\Lambda({\rm d}t)/t$ and let $\varrho:=Q_T$ denote the image of $Q$ under the transformation $T(t):=-\log(1-t)$, $t\in(0,1)$. Then, for all $\eta\ge 0$,
   \[
   \varphi(\eta)
   \ :=\ d\eta+\kappa+\int_{(0,1)}\big(1-(1-t)^\eta\big)\,Q({\rm d}t)
   \ =\ d\eta+\kappa+\int_{(0,\infty)}(1-e^{-\eta u})\,\varrho({\rm d}u).
   \]
   Thus, the map $\eta\mapsto\varphi(\eta)$, $\eta\ge 0$, is the Laplace exponent of a L\'evy process with non-decreasing paths (subordinator) $S=(S_t)_{t\ge 0}$ with drift $d$, killing rate $\kappa$ and L\'evy measure $\varrho$.
   Note that the generator $L^S$ of the subordinator $S$ acts on functions $f\in C^2([0,\infty))$ with $f,f',f''\in C_0([0,\infty))$ via
   \begin{equation} \label{subgen}
      L^Sf(x)\ =\ df'(x) - \kappa f(x) + \int_{(0,\infty)}
      \big(f(x+u)-f(x)\big)\,\varrho({\rm d}u),\qquad x\in[0,\infty).
   \end{equation}
   Obviously, from (\ref{generator_general}) it follows that the generator $L^X$ of the process $X$ obeys the decomposition $L^X=L_1+L_2$, where
   \[
   L_1f(x)\ :=\ -dxf'(x)+\kappa\big(f(0)-f(x)\big)+\int_{(0,1)}\big(f(x(1-t))-f(x)\big)\,Q({\rm d}t)
   \]
   and
   \[
   L_2f(x)\ :=\ d(1-x)f'(x)+\kappa\big(f(1)-f(x)\big)+\int_{(0,1)}\big(f(x(1-t)+t)-f(x)\big)\,Q({\rm d}t).
   \]
   Applying the transformation theorem to (\ref{subgen}) it is readily checked that $L_1$ and $L_2$ are the generators of the processes $(e^{-S_t})_{t\ge 0}$ and $(1-e^{-S_t})_{t\ge 0}$ respectively. Thus, the processes having generators $L_1$ and $L_2$ respectively both have a probabilistic interpretation in terms of the subordinator $S$. There does not seem to be a similar straightforward interpretation of the process $X$ (having generator $L_1+L_2$) in terms of the subordinator $S$.
\end{remark}

Since the state space $[0,1]$ is compact, the limiting process $X$ in Theorem \ref{main1} has a unique stationary distribution $\mu$ (see, for example, Ethier and Kurtz \cite[p.~240, Theorem 9.3 and Remark 9.4]{EthierKurtz2005}). A recursion for the moments
\begin{equation} \label{moments}
   \mu_n\ :=\ \int f_n\,{\rm d}\mu\ =\ \int x^n\,\mu({\rm d}x),\qquad n\in\nz_0,
\end{equation}
of $\mu$ is obtained as follows. By \cite[Proposition 9.2]{EthierKurtz2005} and (\ref{monomial}), for every $n\in\nz_0$,
\[
0
\ =\ \int L^Xf_n\,{\rm d}\mu
\ =\ \int\bigg(\sum_{k=1}^n \varphi(k,n) f_{n-k}-2\varphi(n)f_n\bigg)\,{\rm d}\mu
\ = \ \sum_{k=1}^n \varphi(k,n)\mu_{n-k} - 2\varphi(n)\mu_n.
\]
If $\Lambda$ is not the zero measure, then $\varphi(n)>0$ for all $n\in\nz$ and, hence, the moments satisfy the recursion $\mu_0=1$ and
\begin{equation} \label{momrec}
   \mu_n\ =\ \frac{1}{2\varphi(n)}\sum_{k=1}^n\varphi(k,n)\mu_{n-k},
   \qquad n\in\nz,
\end{equation}
and are hence uniquely determined. Note that the moments $\mu_0,\mu_1,\ldots$ uniquely determine the stationary distribution $\mu$, since the state space $[0,1]$ is compact. The following proposition provides a system of equations, which characterizes the moment generating function (mgf) of the stationary distribution $\mu$.
\begin{proposition}[Characterizing equation for the mgf of the stationary distribution] \label{intrep}
   The moment generating function (mgf) $g:[0,1]\to[0,\infty)$ of the stationary distribution $\mu$ of $X$, defined via $g(x):=\sum_{n=0}^\infty \mu_nx^n/n!$ for all $x\in[0,1]$, satisfies for every $x\in[0,1]$ the
   equation
   \begin{equation} \label{inteqn}
      2dxg'(x)+2\int_{(0,1]}\frac{g(x)-g(x(1-t))}{t}\,\Lambda({\rm d}t)\ =\
      dxg(x)+\int_{(0,1]}\frac{e^{xt}-1}{t}g\big(x(1-t)\big)\,\Lambda({\rm d}t)
   \end{equation}
   with $d:=\Lambda(\{0\})$. Conversely, if $g$ is the mgf of some distribution on $[0,1]$ and if $g$ satisfies
   the equation (\ref{inteqn}) for all $x\in[0,1]$, then $g$ is the mgf of the stationary distribution $\mu$ of $X$.
\end{proposition}
\begin{remark} (Solving Eq.~\eqref{inteqn}).
   If $\Lambda=\delta_0$ is the Dirac measure at $0$, then (\ref{inteqn}) reduces to the differential equation $2g'=g$ with solution $g(x)=e^{x/2}$ corresponding to $\mu=\delta_{1/2}$ as expected. For $\Lambda=\delta_1$, (\ref{inteqn}) immediately yields $g(x)=(e^x+1)/2$ corresponding $\mu=\frac{1}{2}\delta_0+\frac{1}{2}\delta_1$. For general $\Lambda$ it seems to be not straightforward to determine the solution $g$ to (\ref{inteqn}). Equations of the form (\ref{inteqn}) are related to Carleman type integral equations, which could be solved using particular analytic techniques. We refer the reader exemplary to Estrada and Kanwal \cite{EstradaKanwal1987} for such methods. The solutions typically involve residua and Cauchy principle values and are hence rarely known explicitly. We will see later that, for some particular but important classes of measures $\Lambda$, the stationary distribution $\mu$ is known explicitly.
\end{remark}
We now turn to duality results for the limiting process $X$. We first consider the particular case when $\Lambda$ is the uniform distribution, since in this case the duality result is explicit and straightforward to verify. Consider the duality kernel $D:[0,1]\times\nz_0\to\rz$ defined via
\begin{equation} \label{specialkernel}
   D(x,n)\ :=\ (n+1)x^n,\qquad x\in[0,1], n\in\nz_0.
\end{equation}
With this $D$, the following duality result holds.
\begin{proposition}[Duality, \mbox{$\Lambda=U[0,1]$}] \label{dualityspecial}
   Assume that $\Lambda$ is the uniform distribution on $[0,1]$. Then the process $X$ is dual to $Y$ with respect to the duality kernel $D$ defined via (\ref{specialkernel}), where $Y$ is a death process with state space $\nz_0$ and generator
   \begin{equation}
      L^Yf(n)
      \ :=\ \sum_{k=1}^n \bigg(\frac{1}{k}+\frac{1}{n-k+1}\bigg) \big(f(n-k)-f(n)\big),
      \qquad n\in\nz_0.
   \end{equation}
\end{proposition}
Proposition \ref{dualityspecial} generalizes for arbitrary non-zero measure $\Lambda$ as follows.
\begin{theorem}[Duality]
   \label{main2}
   Assume that $\Lambda$ is not the zero measure. Then $X$ is dual to $Y$ with respect to the duality kernel $D:[0,1]\times\nz_0\to\rz$ defined via
   \begin{equation} \label{momkernel}
      D(x,n)\ :=\ \frac{x^n}{\mu_n},\qquad x\in[0,1], n\in\nz_0,
   \end{equation}
   where $Y$ is a death process with state space $\nz_0$, whose generator
   $L^Y$ acts on bounded functions $f:\nz_0\to\rz$ via
   \begin{equation} \label{Ygen}
      L^Yf(n)\ =\ \frac{1}{\mu_n}\sum_{k=1}^n\varphi(k,n)\mu_{n-k}\big(f(n-k)-f(n)\big),
   \qquad n\in\nz_0.
   \end{equation}
\end{theorem}
\begin{remark} (Labelled particles).
   If the death process $Y$ is in state $n\in\nz$, then $k\in\{1,\ldots,n\}$ specific particles are removed at the rate $\frac{\mu_{n-k}}{\mu_n}\int_{[0,1]}t^{k-1}(1-t)^{n-k}\,\Lambda({\rm d}t)$. Note that $0$ is an absorbing state of $Y$.
\end{remark}
\begin{remark} (Explicit duality function).
   Theorem \ref{main2} becomes explicit if the moments $\mu_n$, $n\in\nz_0$, of $\mu$ are known. 
   The most prominent example occurs when $\Lambda$ is the uniform distribution. In this case (see Section \ref{beta}) the stationary distribution $\mu$ is as well the uniform distribution having moments $\mu_n=1/(n+1)$, $n\in\nz_0$. Thus, in this case
   the duality kernel is given by $D(x,n)=(n+1)x^n$, bringing us back to Proposition \ref{dualityspecial}. Section \ref{beta} focusses more generally on the case when $\Lambda=\beta(a,b)$ is the beta distribution with parameters $a,b>0$. In this case it turns out that the stationary distribution $\mu$ is a beta distribution if and only if $a\in\{1,2\}$ or $a=b$. In this case it is the $\beta(b/a,b/a)$-distribution. Section \ref{beta} also provides some information on the stationary distribution $\mu$ for the case $a+b=1$. For all other cases, only little seems to be known on the stationary distribution $\mu$.
\end{remark}
We now turn to an analog duality result for the discrete process $X^{(N)}$. Instead of monomials, now descending factorials come into play. For $m\in\{0,\ldots,N\}$ define $h_m^{(N)}(n):=(n)_m/(N)_m$, $n\in\{0,\ldots,N\}$. Note that $1=h_0^{(N)}\ge h_1^{(N)}\ge\cdots\ge h_N^{(N)}$. The following key lemma is the discrete counterpart of Lemma \ref{generator_monomial}. The proof of Lemma \ref{generator_hypergeometric} is based on an elementary but useful sampling duality relation for binomial distributions, which is deferred to the proof section (see Lemma \ref{aux}) for convenience.
\begin{lemma}[Action of $L^{X^{(N)}}$ on $h_m^{(N)}$] \label{generator_hypergeometric}
   The generator $L^{X^{(N)}}$ acts on $h_m^{(N)}$ via
   \begin{equation} \label{generatorhypergeometric}
      L^{X^{(N)}}h_m^{(N)}
      \ =\ \sum_{k=1}^m \varphi(k,m)h_{m-k}^{(N)} - 2\varphi(m)h_m^{(N)},
      \qquad m\in\{0,\ldots,N\},
   \end{equation}
   where $\varphi(.,.)$ and $\varphi(.)$ are defined via (\ref{lambda_rep}) and
   (\ref{totalrate}).
\end{lemma}
\begin{remark} (Convergence of the stationary distributions).
Clearly, the process $X^{(N)}$ has a stationary distribution, which is denoted by $\mu^{(N)}$ in the following. Define
\[
\mu_m^{(N)}\ :=\ \int h_m^{(N)}\,{\rm d}\mu^{(N)}
\ =\ \sum_{n=0}^N\frac{(n)_m}{(N)_m}\mu^{(N)}(\{n\}),
\qquad m\in\{0,\ldots,N\}.
\]
By \cite[Proposition 9.2]{EthierKurtz2005} and (\ref{generatorhypergeometric}),
\begin{eqnarray*}
   0
   & = & \int L^{X^{(N)}}h_m^{(N)}\,{\rm d}\mu^{(N)}
   \ =\ \int \bigg(\sum_{k=1}^m \varphi(k,m)h_{m-k}^{(N)}-2\varphi(m)h_m^{(N)}\bigg)\,{\rm d}\mu^{(N)}\\
   & = & \sum_{k=1}^m \varphi(k,m)\mu_{m-k}^{(N)} - 2\varphi(m)\mu_m^{(N)}.
\end{eqnarray*}
Thus, if $\Lambda\ne 0$, then the $\mu_m^{(N)}$ satisfy the recursion $\mu_0^{(N)}=1$ and
\begin{equation} \label{recursionhyper}
   \mu_m^{(N)}\ =\ \frac{1}{2\varphi(m)}\sum_{k=1}^m \varphi(k,m)\mu_{m-k}^{(N)},
\qquad m\in\{1,\ldots,N\}.
\end{equation}
Comparing (\ref{recursionhyper}) with the recursion (\ref{momrec}) it follows that $\mu_m^{(N)}=\mu_m$ if $m\in\{0,\ldots,N\}$ and $\mu_m^{(N)}=0$ otherwise.
In particular, for all $m\in\{0,\ldots,N\}$, $\mu_m^{(N)}=\mu_m$ does not depend on $N$, a crucial property simplifying the further analysis considerably. For example, for all $m\in\nz_0$, $\mu_m^{(N)}\to\mu_m$ as $N\to\infty$. This convergence of moments implies the convergence $Z_N/N\to Z$ in distribution as $N\to\infty$, where $Z_N$ and $Z$ are random variables having distribution $\mu^{(N)}$ and $\mu$ respectively.
\end{remark}
\color{black}
\begin{theorem}[Duality for the process $X^{(N)}$]
   \label{main3}
   Assume that $\Lambda$ is not the zero measure. Then the process $X^{(N)}$ is dual to the process $Y^{(N)}$ with respect to the duality kernel $D^{(N)}:\{0,\ldots,N\}^2\to\rz$ defined via
   \begin{equation} \label{hyperkernel}
      D^{(N)}(n,m)\ :=\ \frac{h_m^{(N)}(n)}{\mu_m^{(N)}}
      \ =\ \frac{(n)_m}{(N)_m\mu_m^{(N)}}
        \ =\ \frac{\binom{n}{m}}{\binom{N}{m}\mu_m^{(N)}},
      \qquad n,m\in\{0,\ldots,N\},
   \end{equation}
   where $Y^{(N)}$ is a death process with state space $\{0,\ldots,N\}$, whose generator $L^{Y^{(N)}}$ acts on functions $f:\{0,\ldots,N\}\to\rz$ via
   \begin{equation} \label{YNgen}
      L^{Y^{(N)}}f(m)\ =\ \frac{1}{\mu_m^{(N)}}\sum_{k=1}^m \varphi(k,m)\mu_{m-k}^{(N)}\big(f(m-k)-f(m)\big),
      \qquad m\in\{0,\ldots,N\}.
   \end{equation}
\end{theorem}
\begin{remark} ($Y^{(N)}$ as a restriction of $Y$).
   The right hand side in (\ref{YNgen}) does not depend on $N$, since $\mu_m^{(N)}=\mu_m$ for all $m\in\{0,\ldots,N\}$ as observed after (\ref{recursionhyper}). The process $Y^{(N)}$ thus depends on $N$ only via its state space $\{0,\ldots,N\}$. One may interpret the process $Y^{(N)}$ simply as the process $Y$ restricted to the state space $\{0,\ldots,N\}$.
\end{remark}
\color{black}
We finish our results for general $\Lambda$ with a convergence result for the discrete dual process $Y^{(N)}$.
\begin{theorem}[Convergence of the dual processes]
   \label{main4}
   Assume that $\Lambda$ is not the zero measure and that $Y_0^{(N)}\to Y_0$ in distribution as $N\to\infty$ for some random variable $Y_0$. Then, as $N\to\infty$, the discrete dual process $Y^{(N)}$ with generator (\ref{YNgen}) converges in $D_{\nz_0}([0,\infty))$ to the process $Y$ with generator (\ref{Ygen}).
\end{theorem}
The commutative diagram in Figure \ref{fig1} illustrates these convergence and duality results obtained so far.
\begin{figure}[ht]
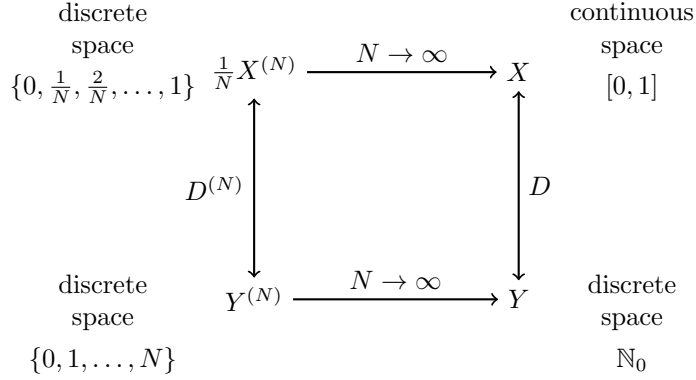

\centering
\scaling
\caption{Commutative diagram for the process $X^{(N)}$ and the death process $Y^{(N)}$. The two right-arrows `$\rightarrow$' stand for convergence in $D_{[0,1]}([0,\infty))$ and $D_{\nz_0}([0,\infty))$ respectively as $N\to\infty$. The vertical updown-arrows `$\updownarrow$' stand for duality, on the left hand side with respect to the sampling duality kernel $D^{(N)}(n,m)=(n)_m/((N)_m\mu_m^{(N)})$ and on the right hand side with respect to the moment duality kernel $D(x,n)=x^n/\mu_n$.}
\label{fig1}
\end{figure}

\subsubsection{A more advanced duality result} \label{advanced}

We now turn to a more involved duality result. Having the particular duality statement in Proposition \ref{dualityspecial} for the uniform distribution $\Lambda=\beta(1,1)$ in mind, it is natural to consider for general finite measure $\Lambda$ the Markov process $Y$ with state space $\nz_0$ and generator
\begin{equation} \label{generatory}
   L^Yf(n)
   \ :=\ \sum_{k=1}^n\big(\varphi(k,n)+\varphi(n-k+1,n)\big) \big(f(n-k)-f(n)\big),
   \quad n\in\nz_0,
\end{equation}
where $\varphi(.,.)$ is defined via (\ref{lambda_rep}).
The following statement holds.
\begin{theorem}[Polynomial duality with dual process \eqref{generatory}] \label{duality}
   Let $\Lambda$ be a finite measure on $[0,1]$. Furthermore, let $r:=(r_n)_{n\in\nz_0}$ be a given sequence of real numbers. Then there exists a unique duality kernel $D=D_r:[0,1]\times\nz_0\to\rz$ (depending on the sequence $r$) satisfying the following three properties.
   \begin{enumerate}
      \item[(i)] The kernel $D$ is of polynomial form
         \begin{equation} \label{polynomial}
            D(x,n)\ =\ \sum_{j=0}^n d_{n,j}x^j,\qquad x\in[0,1], n\in\nz_0,
         \end{equation}
         for some real coefficients $d_{n,j}$, $n\in\nz_0$, $j\in\{0,\ldots,n\}$.
      \item[(ii)] The diagonal coefficients are given by $d_{n,n}=r_n$ for all $n\in\nz_0$.
      \item[(iii)] $X$ is dual to $Y$ with respect to $D$, where $Y$ is the process with generator (\ref{generatory}).
   \end{enumerate}
\end{theorem}
\begin{remark} (Explicit polynomial duality).
   For general measure $\Lambda$ the coefficients $d_{n,j}$, $n\in\nz_0$, $j\in\{0,\ldots,n\}$, do not have a simple form. One exception is uniform distribution $\Lambda=\beta(1,1)$. In this case, choosing the diagonal entries as $d_{n,n}=n+1$ ($=:r_n$), the recursion (\ref{recursion}) yields $d_{n,j}=0$ for all $0\le j<n$, and, hence, Theorem \ref{duality} boils down to Proposition \ref{dualityspecial}. For $\Lambda\ne\beta(1,1)$ it is impossible to choose the diagonal entries $d_{n,n}$, $n\in\nz_0$, in such a way that all the off-diagonal entries $d_{n,j}$, $j<n$, vanish.
\end{remark}
\begin{remark} (Space of duality functions).
   Theorem \ref{duality} shows that for each sequence $r=(r_n)_{n\in\nz_0}$
   there exists a duality kernel $D=D_r$ such that $X$ is dual to $Y$ with respect to $D$. It can be expected that the kernels $D_{e_i}$, where $e_i$ denotes the $i$th unit vector in $\rz^{\nz_0}$, are linear independent elements of the Banach space $B([0,1]\times\nz_0)$ of all measurable, real-valued, bounded functions on $[0,1]\times\nz_0$. It is therefore conjectured that the space $\{D\in B([0,1]\times\nz_0)\,:\,\mbox{$X$ is dual to $Y$ with respect to $D$}\}$ of all duality kernels has infinite dimension.
\end{remark}

\subsubsection{The beta-mutation model} \label{beta}
We focus in this section on the particular subclass of models when the measure $\Lambda$ is the beta distribution $\beta(a,b)$ with parameters $a,b>0$. Clearly, all results obtained so far are applicable. For example, the generator $L^X$ of the limiting Markov process $X=(X_t)_{t\ge 0}$ in Theorem \ref{main1} acts on functions $f\in C^2([0,1])$ via
\begin{equation} \label{generator}
   L^Xf(x)
   \ =\ \frac{1}{{\rm B}(a,b)}\int_0^1
   \big(f(x(1-t)) + f(x(1-t)+t) - 2f(x)\big)\frac{t^{a-1}(1-t)^{b-1}}{t}\,{\rm d}t,
\end{equation}
where ${\rm B}(.,.)$ denotes the beta function. For obvious reasons we call this model the beta-mutation model. The rates (\ref{lambda_rep}) have the form
\begin{equation} \label{rate}
   \varphi_{a,b}(k,n)
   \ :=\ \binom{n}{k}\frac{{\rm B}(k+a-1,n-k+b)}{{\rm B}(a,b)},
   \qquad n\in\nz,k\in\{1,\ldots,n\}
\end{equation}
and the total rates (\ref{totalrate}) reduce to
\begin{eqnarray}
\label{total-rate-ab}
   \varphi_{a,b}(n)
   & = & \frac{1}{{\rm B}(a,b)}\int_0^1 \frac{1-(1-t)^n}{t}t^{a-1}(1-t)^{b-1}\,{\rm d}t\nonumber\\
   & = & \left\{
      \begin{array}{cl}
         \displaystyle\frac{{\rm B}(a-1,b)-{\rm B}(a-1,n+b)}{{\rm B}(a,b)} & \mbox{if $a\ne 1$,}\\
         b\big(\Psi(n+b)-\Psi(b)\big) & \mbox{if $a=1$,}
      \end{array}
   \right. \label{totalratebeta}
\end{eqnarray}
where $\Psi:=(\log\Gamma)'=\Gamma'/\Gamma$ denotes the digamma function
(logarithmic derivative of the gamma function). Note that ${\rm B}(a-1,b)-{\rm B}(a-1,n+b)\to\Psi(n+b)-\Psi(b)$ as $a\to 1+$. Thus, the formula for $a=1$ is the continuous extension of the formula for $a>1$. For $0<a<1$, the negative argument $a-1\in(-1,0)$ occurring in the beta function of the above formula does not cause problems, since ${\rm B}(p,q)$ is defined for all $p,q\in\rz\setminus\{\ldots,-2,-1,0\}$ by analytic continuation.

The map $\eta\mapsto\varphi_{a,b}(\eta)$, $\eta\ge 0$, is the Laplace exponent of a drift-free subordinator, whose L\'evy measure $\varrho$ has density $u\mapsto e^{-bu}(1-e^{-u})^{a-2}/{\rm B}(a,b)$, $u>0$, with respect to the Lebesgue measure on $(0,\infty)$.

For $(a,b)=(1,2s)$ with $s>0$, the process $X$ and the corresponding rates $\varphi_s(k,n):=\varphi_{1,2s}(k,n)$ play a certain role in the mathematical statistical physics literature on particle systems. We refer the reader to the work of Frassek, Giardin\`a and Kurchan \cite[Eq.~(2.30)]{FrassekGiardinaKurchan2020} and to further studies on this model \cite{FrassekGiardina2022,FranceschiniFrassekGiardina2023, CarinciFranceschiniFrassekGiardinaRedig2023, GiardinaRedigvanTol2025,CarinciFranceschiniGabrielliGiardinaTsagkarogiannis2024,RedigvanTol2026}. In this case the model is called `harmonic' because of the occurrence of the harmonic factor $t^{-1}$
below the integral in (\ref{generator}) and since $\varphi_{\frac{1}{2}}(k,n)=\varphi_{1,1}(k,n)=1/k$.

For $(a,b)=(1-\alpha,\alpha)$ for some $\alpha\in(0,1)$, the generator
$L^X$ occurs (up to a factor) as a mutation generator in the mathematical population genetics literature (see, for example, Handa \cite[Eq.~(1.3)]{Handa2014} with $c_1:=c_2:=1$).

The following proposition clarifies when the stationary distribution $\mu$ of the beta-mutation model is a beta distribution.
\begin{proposition}{(Stationary distribution of the beta-mutation model).} \label{betastat}
   Let $\Lambda=\beta(a,b)$ with $a,b>0$. Then the following two statements hold.
   \begin{enumerate}
      \item[(i)] The stationary distribution $\mu$ of $X$ is a beta distribution if and only if $a\in\{1,2\}$ or $a=b$. In this case $\mu$ is the symmetric $\beta(b/a,b/a)$-distribution.
      \item[(ii)] The stationary distribution $\mu$ of $X$ is reversible if and only if $a=1$.
   \end{enumerate}
\end{proposition}
\begin{example} (Duality for $\Lambda=\beta(1,b)$).
   Assume that $a=1$. By Proposition \ref{betastat} (i), the stationary
   distribution $\mu$ of $X$ is the beta distribution $\beta(b,b)$ having moments
   \[
   \mu_n\ :=\ \frac{{\rm B}(n+b,b)}{{\rm B}(b,b)}
   \ =\ \frac{\Gamma(2b)\Gamma(n+b)}{\Gamma(b)\Gamma(n+2b)},
   \qquad n\in\nz_0.
   \]
   By Theorem \ref{main2}, the process $X$ is dual to $Y$ with respect to the duality kernel $D$ defined via
   \[
   D(x,n)\ =\ \frac{x^n}{\mu_n}
   \ =\ \frac{\Gamma(b)\Gamma(n+2b)}{\Gamma(2b)\Gamma(n+b)}x^n,\qquad x\in[0,1],n\in\nz_0,
   \]
   where $Y$ is a death process with state space $\nz_0$, whose generator $L^Y$ acts on bounded functions $f:\nz_0\to\rz$ via
   \[
   L^Yf(n)\ :=\ \frac{n!\Gamma(n+2b)}{(\Gamma(n+b))^2}
   \sum_{k=1}^n\frac{(\Gamma(n-k+b))^2}{k(n-k)!\Gamma(n-k+2b)}
   \big(f(n-k)-f(n)\big),\qquad n\in\nz_0.
   \]
\end{example}
\begin{remark} (Spectral decomposition).
   In the reversible case, i.e. the $\beta(1,b)$-model, it is possible to give a spectral decomposition for the generator of the process $X^{(N)}$, as well as for the process $X$. The eigenvectors are given by Hahn orthogonal polynomials in the first case, and Jacobi orthogonal polynomials in the second case. Because $a=1$ is the only reversible case, this class is essentially the only model where the spectral decomposition can be obtained explicitly. The spectral decomposition immediately entails a closed formula for the semigroup of the processes. We shall report more about this in a future work.
\end{remark}
The remaining part of this section deals with the $\beta(a,b)$-mutation model for the case $a+b=1$. Define the generalized Stieltjes transform $S_b$ of the stationary distribution $\mu$ of the process $X$ via $S_b(t):=\int (1+tx)^{-b}\,\mu({\rm d}x)$ for all $t\ge 0$.
\begin{lemma}[$a+b=1$, Stieltjes transform] \label{stieltjes}
   If $a+b=1$, then the generalized Stieltjes transform $S_b$ of the stationary distribution $\mu$ is given by
   \begin{equation} \label{S}
      S_b(t)\ =\ \frac{2}{1+(1+t)^b},\qquad t\ge 0.
   \end{equation}
\end{lemma}
\begin{lemma}[$a+b=1$, stationary distribution] \label{aplusbequal1}
   If $a+b=1$, then the stationary distribution $\mu$ of the process $X$ has a density $f_\mu$ with respect to the Lebesgue measure on $(0,1)$, which is given by
   \begin{equation} \label{density}
      f_\mu(x)\ =\ \frac{\textcolor{black}{b}}{\pi i x^2}
      \int_C \frac{1}{(1+w)^{1-\textcolor{black}{b}}(1+\frac{w}{x})^{1-\textcolor{black}{b}}\big(1+(1+\frac{w}{x})^{\textcolor{black}{b}}\big)^2}
      \,{\rm d}w,\qquad x\in(0,1),
   \end{equation}
   where $C$ is a counterclockwise contour that starts and ends at the point $w=-1$ and encloses the origin.
\end{lemma}
\begin{remark} (Contour integral).
   The contour $C$ could be chosen to be the counterclockwise unit circle for instance. The function below the integral in (\ref{density}) has three singularities, one at $w=-1$, one at $w=-x$ and another one at $w=x(e^{i\pi /b}-1)$. We leave a further evaluation of the contour integral in (\ref{density}) for future studies.
\end{remark}
The following proposition provides formulas for the moments $\mu_n:=\int x^n\mu({\rm d}x)$, $n\in\nz_0$, of the stationary distribution of the process $X$.
\begin{proposition}[$a+b=1$, moments]  \label{specialmoments}
   If $a+b=1$ then the moments of the stationary distribution $\mu$ are given by
   \begin{equation} \label{mom}
   \mu_n
    = \frac{1}{\binom{-b}{n}}
        \sum_{k=1}^n\frac{1}{2^k}\sum_{\ell=1}^k(-1)^\ell\binom{k}{\ell}
        \binom{b\ell}{n}
    = \frac{1}{(-b)_n}\sum_{k=1}^n\frac{(-1)^k}{2^k}k!
        \sum_{j=1}^n s(n,j)b^jS(j,k),\ n\in\nz,
   \end{equation}
   where $s(.,.)$ and $S(.,.)$ denote the Stirling numbers of the first and second kind respectively.
\end{proposition}
\begin{remark} (Relation to generalized Stirling numbers)
   The sum over $j$ arising in (\ref{mom}) is well-known to be of the form
   $\sum_{j=1}^n s(n,j)b^jS(j,k)=b^nS(n,k;1/b,1,0)=b^kS(n,k;1,b,0)$,   
   where $S(n,k;a,b,r)$ denote the generalized Stirling numbers studied by Hsu and Shiue \cite{HsuShiue1998}. Thus, the $n$th moment is also given by
   $\mu_n=\frac{1}{(-b)_n}\sum_{k=1}^n(-\frac{b}{2})^kk!S(n,k;1,b,0)$, $n\in\nz$.
\end{remark}
\begin{remark} (Additional comments).
   The first four moments are $\mu_1=\frac{1}{2}$, $\mu_2=\frac{1}{2(1+b)}$,
   $\mu_3:=\frac{2-b}{4(1+b)}$ and $\mu_4:=\frac{3(2-b^2)}{2(b+1)(b+2)(b+3)}$. Except for $b=\frac{1}{2}$, the fourth moment does not coincide with the fourth moment of the $\beta(b/a,b/a)$-distribution, showing that for $b\ne\frac{1}{2}$ the stationary distribution cannot be a beta distribution. As expected, for $b\to0$ one obtains $\mu_n=\frac{1}{2}$ for all $n\in\nz$ corresponding to the stationary distribution $\mu=\frac{1}{2}\delta_0+\frac{1}{2}\delta_1$. Similarly, for $b\to 1$ one obtains $\mu_n=2^{-n}$ for all $n\in\nz$ corresponding to $\mu=\delta_{\frac{1}{2}}$.
\end{remark}
The results of this section are illustrated in Figure \ref{figure2}.
\begin{figure}[ht]
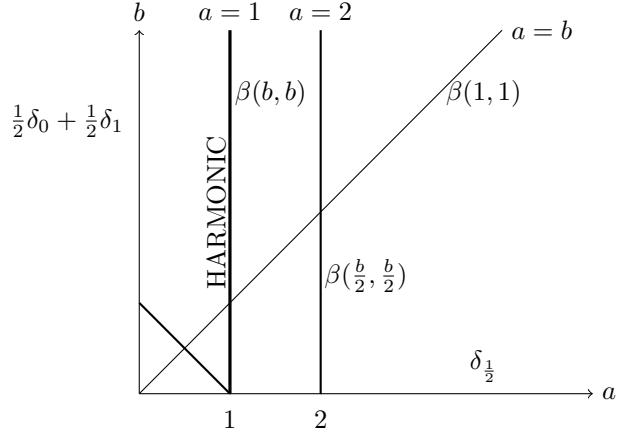

\centering
\betapicture
\caption{A graphical illustration of the results for the case when $\Lambda=\beta(a,b)$ is the beta distribution with parameters $a,b>0$. If $a=1$ then $\beta(b,b)$ is reversible. If $a+b=1$ then $\mathbb{E}(1/(1+tX_{\infty})^b)=2/(1+(1+t)^b)$.}
\label{figure2}
\end{figure}

\subsection{Proofs} \label{proofs}

\begin{proof}[Proof of Proposition \ref{equivalent}]
   Let us first verify the equivalence of (i) and (ii).

   \vspace{2mm}

   (i) $\Rightarrow$ (ii): If the array $(\varphi(k,n))_{n\in\nz,k\in\{1,\ldots,n\}}$ satisfies the consistency relations (\ref{consis}), then the modified numbers $\lambda(k,n):=\varphi(k,n)/\binom{n}{k}$ obviously satisfy the modified consistency relations $\lambda(k,n)=\lambda(k,n+1)+\lambda(k+1,n+1)$, $n\in\nz$, $k\in\{1,\ldots,n\}$. Now proceed as in the proof of Pitman \cite[Lemma 18]{Pitman1999}. If $\varphi(1,1)=0$, then (\ref{lambda_rep}) holds with $\Lambda:\equiv 0$. Assume now that $\varphi(1,1)>0$ and define $\mu(i,j):=\lambda(i+1,i+j+1)/\lambda(1,1)$ for all $i,j\in\nz_0$. Then, $\mu(0,0)=1$ and, for all $i,j\in\nz_0$,
   \begin{eqnarray*}
      \mu(i+1,j)+\mu(i,j+1)
      & = & \frac{\lambda(i+2,i+j+2)+\lambda(i+1,i+j+2)}{\lambda(1,1)}\\
      & = & \frac{\lambda(i+1,i+j+1)}{\lambda(1,1)}\ =\ \mu(i,j).
   \end{eqnarray*}
   De Finetti's representation for infinite exchangeable sequences of Bernoulli random variables implies that there exists a random variable $S$ taking values in $[0,1]$ such that $\mu(i,j)=\me(S^i(1-S)^j)$ for all $i,j\in\nz_0$. Thus, for all $n\in\nz$ and $k\in\{1,\ldots,n\}$, $\lambda(k,n)=\lambda(1,1)\mu(k-1,n-k)=\lambda(1,1)\me(S^{k-1}(1-S)^{n-k})$. Hence, (\ref{lambda_rep}) holds with $\Lambda:=\lambda(1,1)\pr_S$, where $\pr_S$ denotes the distribution of $S$.

   \vspace{2mm}

   (ii) $\Rightarrow$ (i): Conversely, assume that (\ref{lambda_rep}) holds for some finite measure $\Lambda$ on $[0,1]$. Then, a straightforward calculation shows that the numbers $\lambda(k,n):=\varphi(k,n)/\binom{n}{k}=\int_{[0,1]}t^{k-1}(1-t)^{n-k}\,\Lambda({\rm d}t)$ satisfy the relations $\lambda(k,n)=\lambda(k,n+1)+\lambda(k+1,n+1)$, $n\in\nz$, $k\in\{1,\ldots,n\}$. Multiplication with $\binom{n}{k}$ shows that the numbers $\varphi(n,k)$ satisfy the consistency relations (\ref{consis}).

   \vspace{2mm}

   We now verify the equivalence of (i) and (iii).

   \vspace{2mm}

   (i) $\Leftrightarrow$ (iii): We consider the generator $L$ acting on bounded test functions $f:\nz_0^2\to\rz$ via (\ref{twodimgen}) and show that the consistency relations (\ref{consis}) for the rates are equivalent to
   \begin{equation} \label{3}
      [L,A]\ =\ 0,
   \end{equation}
   where $A$ denotes the annihilation operator acting as
   \begin{equation} \label{4}
      A f(n_1,n_2)\ =\ n_1 f(n_1-1,n_2) + n_2 f(n_1,n_2-1).
   \end{equation}
   We will need the following:
   \begin{eqnarray} \label{5}
      Lf(n_1-1,n_2)
      & = & \sum_{k=1}^{n_1-1}\varphi(k,n_1-1)
            \big(f(n_1-1-k,n_2+k)-f(n_1-1,n_2)\big)\nonumber\\
      & + & \sum_{k=1}^{n_2}\varphi(k,n_2)
            \big(f(n_1-1+k,n_2-k)-f(n_1-1,n_2)\big),
   \end{eqnarray}
   \begin{eqnarray} \label{6}
      Lf(n_1,n_2-1)
      & = & \sum_{k=1}^{n_1}\varphi(k,n_1)
            \big(f(n_1-k,n_2-1+k)-f(n_1,n_2-1)\big)\nonumber\\
      & + & \sum_{k=1}^{n_2-1}\varphi(k,n_2-1)
            \big(f(n_1+k,n_2-1-k) - f(n_1,n_2-1)\big),
   \end{eqnarray}
   as well as
   \begin{equation} \label{7}
      Af(n_1-k,n_2+k)\ =\ (n_1-k)f(n_1-k-1,n_2+k)+(n_2+k)f(n_1-k,n_2+k-1)
   \end{equation}
   and
   \begin{equation} \label{8}
      Af(n_1+k,n_2-k)\ =\ (n_1+k)f(n_1+k-1,n_2-k)+(n_2-k) f(n_1+k,n_2-k-1).
   \end{equation}
   We start by computing the action of $LA$. Using \eqref{twodimgen} we have
   \begin{eqnarray} \label{9}
      LAf(n_1,n_2)
      & = & \sum_{k=1}^{n_1}\varphi(k,n_1)
            \big(Af(n_1-k,n_2+k)-Af(n_1,n_2)\big)\nonumber\\
      & + & \sum_{k=1}^{n_2}\varphi(k,n_2)
            \big(Af(n_1+k,n_2-k)-Af(n_1,n_2)\big).
   \end{eqnarray}
   Inserting \eqref{4}, \eqref{7} and \eqref{8} into the previous equation, we get
   \begin{eqnarray} \label{10}
      LAf(n_1,n_2)
      & = & \sum_{k=1}^{n_1}\varphi(k,n_1)(n_1-k)f(n_1-k-1,n_2+k)\nonumber\\
      & + & \sum_{k=1}^{n_1}\varphi(k,n_1)(n_2+k) f(n_1-k,n_2+k-1)\nonumber\\
      & - & \sum_{k=1}^{n_1}\varphi(k,n_1)
            \big(n_1f(n_1-1,n_2)+n_2f(n_1,n_2-1)\big)\nonumber\\
      & + & \sum_{k=1}^{n_2}\varphi(k,n_2)(n_1+k)f(n_1+k-1,n_2-k)\nonumber\\
      & + & \sum_{k=1}^{n_2}\varphi(k,n_2)(n_2-k)f(n_1+k,n_2-k-1)\nonumber\\
      & - & \sum_{k=1}^{n_2}\varphi(k,n_2)
            \big( n_1 f(n_1-1,n_2) + n_2 f(n_1,n_2-1)\big).
   \end{eqnarray}
   We now compute the action of $AL$. Using \eqref{4} we may write
   \begin{equation} \label{11}
      ALf(n_1,n_2)\ =\ n_1Lf(n_1-1,n_2)+n_2Lf(n_1,n_2-1).
   \end{equation}
   Inserting \eqref{5} and \eqref{6} into the previous equation, we get
   \begin{eqnarray} \label{12}
      ALf(n_1,n_2)
      & = & n_1\sum_{k=1}^{n_1-1}\varphi(k,n_1-1)
            \big(f(n_1-1-k,n_2+k)-f(n_1-1,n_2)\big)\nonumber\\
      & + & n_1\sum_{k=1}^{n_2}\varphi(k,n_2)
            \big(f(n_1-1+k,n_2-k)-f(n_1-1,n_2)\big)\nonumber\\
      & + & n_2\sum_{k=1}^{n_1}\varphi(k,n_1)
            \big(f(n_1-k,n_2-1+k)-f(n_1,n_2-1)\big)\nonumber\\
      & + & n_2\sum_{k=1}^{n_2-1}\varphi(k,n_2-1)
            \big(f(n_1+k,n_2-1-k)-f(n_1,n_2-1)\big).
   \end{eqnarray}
   Now we compare \eqref{10} with \eqref{12}. We fix $k$ with $2\le k\le n_1$ and compare the coefficients of $f(n_1-k,n_2+k-1)$. We get the equation
   \begin{equation} \label{13}
      \varphi(k,n_1)(n_2+k)+\varphi(k-1,n_1)(n_1-k+1)
      \ =\ n_1\varphi(k-1,n_1-1)+n_2\varphi(k,n_1).
   \end{equation}
   The term $n_2\varphi(k,n_1)$ drops and we may rewrite the previous equation as
   \begin{equation} \label{14}
      \varphi(k-1,n_1-1)
      \ =\ \bigg(1-\frac{k-1}{n_1}\bigg)\varphi(k-1,n_1)
      +\frac{k}{n_1}\varphi(k,n_1),\qquad 2\le k \le n_1.
   \end{equation}
   If we now define $k':=k-1$ and $n_1':=n_1-1$, then we have
   \begin{equation} \label{15}
      \varphi(k',n_1')
      \ =\ \bigg(1-\frac{k'}{n_1'+1}\bigg)\varphi(k',n_1'+1)
           +\frac{k'+1}{n_1'+1}\varphi(k'+1,n_1'+1),\qquad 1\le k' \le n_1',
   \end{equation}
   which is the consistency equation \eqref{consis}.

   \vspace{2mm}

   It remains to verify the equivalence of (i) and (iv).
   We essentially proceed as in the proof of the equivalence of (i) and (iii), but nevertheless provide the details because of the occurrence of the additional parameter $N$ in the calculations.

   \vspace{2mm}

   (i) $\Leftrightarrow$ (iv): For $N\in\nz$ we consider the generator $L^{(N)}:=L^{X^{(N)}}$ acting on test functions $f:E_N\to\rz$ via \eqref{onedimgen} and show that the consistency relations \eqref{consis} for the rates are equivalent to
   \begin{equation} \label{18}
      [L^{(N)},A^{(N)}]\ =\ 0\quad\mbox{for all $N\in\nz$},
   \end{equation}
   where $A^{(N)}$ denotes the annihilation operator acting as
   \begin{equation} \label{19}
      A^{(N)}f(n)\ =\ nf(n-1)+(N-n)f(n).
   \end{equation}
   We will need the following:
   \begin{eqnarray} \label{20}
      L^{(N)}f(n-1)
      & = & \sum_{k=1}^{n-1}\varphi(k,n-1)
            \big(f(n-1-k)-f(n-1)\big)\nonumber\\
      & + & \sum_{k=1}^{N-n+1}\varphi(k,N-n+1)
            \big(f(n-1+k)-f(n-1)\big),
   \end{eqnarray}
   as well as
   \begin{equation} \label{21}
      A^{(N)}f(n-k)\ =\ (n-k)f(n-k-1)+(N-n+k)f(n-k)
   \end{equation}
   and
   \begin{equation} \label{22}
      A^{(N)}f(n+k)\ =\ (n+k)f(n+k-1)+(N-n-k)f(n+k).
   \end{equation}
   We start by computing the action of  $L^{(N)}A^{(N)}$. Using \eqref{onedimgen} we have
   \begin{eqnarray} \label{23}
      L^{(N)}A^{(N)}f(n)
      & = & \sum_{k=1}^n\varphi(k,n)\big(A^{(N)}f(n-k)-A^{(N)}f(n)\big)\nonumber\\
      & + & \sum_{k=1}^{N-n} \varphi(k,N-n)\big(A^{(N)}f(n+k)-A^{(N)}f(n)\big).
   \end{eqnarray}
   Inserting \eqref{19}, \eqref{21} and \eqref{22} into the previous equation, we get
   \begin{eqnarray} \label{24}
      L^{(N)}A^{(N)}f(n)
      & = & \sum_{k=1}^n\varphi(k,n)(n-k)f(n-k-1)\nonumber\\
      & + & \sum_{k=1}^n\varphi(k,n)(N-n+k)f(n-k)\nonumber\\
      & - & \sum_{k=1}^n\varphi(k,n)\big(nf(n-1)+(N-n)f(n)\big)\nonumber\\
      & + & \sum_{k=1}^{N-n}\varphi(k,N-n)(n+k)f(n+k-1)\nonumber\\
      & + & \sum_{k=1}^{N-n}\varphi(k,N-n)(N-n-k)f(n+k)\nonumber\\
      & - & \sum_{k=1}^{N-n}\varphi(k,N-n)\big(nf(n-1)+(N-n)f(n)\big).
   \end{eqnarray}
   We now compute the action of $A^{(N)}L^{(N)}$. Using \eqref{19} we may write
   \begin{equation} \label{25}
      A^{(N)}L^{(N)}f(n)\ =\ nL^{(N)}f(n-1)+(N-n)L^{(N)}f(n).
   \end{equation}
   Inserting \eqref{onedimgen} and \eqref{20} into the previous equation, we get
   \begin{eqnarray} \label{26}
      A^{(N)}L^{(N)}f(n)
      & = & n\sum_{k=1}^{n-1}\varphi(k,n-1)
            \big(f(n-1-k)-f(n-1)\big)\nonumber\\
      & + & n\sum_{k=1}^{N-n+1}\varphi(k,N-n+1)
            \big(f(n-1+k)-f(n-1)\big)\nonumber\\
      & + & (N-n)\sum_{k=1}^{n}\varphi(k,n)
            \big(f(n-k)-f(n)\big)\nonumber\\
      & + & (N-n)\sum_{k=1}^{N-n}\varphi(k,N-n)
            \big(f(n+k)-f(n)\big).
   \end{eqnarray}
   Now we compare \eqref{24} with \eqref{26}. We fix $k$ with $2\le k\le n$ and compare the coefficients of $f(n-k)$. We get the equation
   \begin{equation}
      \varphi(k,n)(N-n+k)+\varphi(k-1,n)(n-k+1)
      \ =\ n\varphi(k-1,n-1)+(N-n)\varphi(k,n).
   \end{equation}
   The term $(N-n)\varphi(k,n)$ drops and we may rewrite the previous equation as
   \begin{equation}
      \varphi(k-1,n-1)
      \ =\ \bigg(1-\frac{k-1}{n}\bigg)\varphi(k-1,n)
           +\frac{k}{n}\varphi(k,n),\qquad 2\le k\le n.
   \end{equation}
   If we now define $k':=k-1$ and $n':=n-1$, then we have
   \begin{equation}
      \varphi(k',n')
      \ =\ \bigg(1-\frac{k'}{n'+1}\bigg)\varphi(k',n'+1)
           +\frac{k'+1}{n'+1}\varphi(k'+1,n'+1),\qquad 1\le k'\le n',
   \end{equation}
   which is the consistency equation \eqref{consis}.
\end{proof}
\begin{proof}[Proof of Theorem \ref{main1}]
   Let us first verify that the space $D$ of all polynomials $p:[0,1]\to\rz$ is a core for the generator $L:=L^X$. Clearly, $D$ is dense in the space $H$ of all continuous functions $f:[0,1]\to\rz$. For $n\in\nz_0$ let $D_n$ denote the space of all polynomials of degree less than or equal to $n$. Then $D_0,D_1,\ldots$ is a sequence of finite-dimensional subspaces of $D$ with $\bigcup_{n\in\nz_0}D_n=D$. From (\ref{monomial}) it follows that $LD_n\subseteq D_n$ for all $n\in\nz_0$. We have $(\lambda-L)(D_n)=D_n$ for all $\lambda$ not belonging to the set of eigenvalues of $L|D_n$, i.e. for all but at most finitely many $\lambda>0$. Thus, $(\lambda-L)(D)=(\lambda-L)(\bigcup_{n\in\nz_0}D_n)=\bigcup_{n\in\nz_0} D_n=D$ is dense in $H$ for all but at most countably many $\lambda>0$. By \cite[p.~17, Proposition 3.1]{EthierKurtz2005}, $D$ is a core for $L$.

   Let us now turn to the generator $L^{X^{(N)}}$ of the process $X^{(N)}$.
   Recall that $d:=\Lambda(\{0\})$. For $n\in\nz_0$ and $p\in[0,1]$ let $K_{n,p}$ be a random variable having a binomial distribution with parameters $n$ and $p$. Plugging (\ref{lambda_rep}) into (\ref{onedimgen}), a straightforward calculation shows that, for all $N\in\nz$, $n\in\{0,\ldots,N\}$ and all functions $f:\{0,\ldots,N\}\to\rz$,
   \begin{eqnarray} \label{gen_rep}
      &   & \hspace{-15mm}L^{X^{(N)}}f(n)
      \ = \ dn\big(f(n-1)-f(n)\big)
          +\int_{(0,1]}\frac{\me\big(f(K_{n,1-t})\big)-f(n)}{t}\,\Lambda({\rm d}t)\nonumber\\
      &   & +\ d(N-n)\big(f(n+1)-f(n)\big)
          +\int_{(0,1]}\frac{\me\big(f(N-K_{N-n,1-t})\big)-f(n)}{t}\,\Lambda({\rm d}t).
   \end{eqnarray}
   Now, let $L_N$ denote the generator of the space-scaled process $X^{(N)}/N$ having state space $E_N:=\{j/N:j\in\{0,\ldots,N\}\}$. For any function $f:[0,1]\to\rz$ let the function $\pi_Nf:E_N\to\rz$ be defined by $\pi_Nf(x):=f(x)$ for all $x\in E_N$. To prove convergence in $D_{[0,1]}([0,\infty))$, it suffices by standard theory on weak convergence of Markov processes (see, for example, Ethier and Kurtz \cite[p.~28, Theorem 6.1, p.~28 and p.~232, Corollary 8.7]{EthierKurtz2005} to verify that, for every $f\in D$,
   \begin{equation} \label{uniconv}
      \|L_N(\pi_Nf)-\pi_N(Lf)\|\ :=\
      \sup_{x\in E_N}|L_N(\pi_Nf)(x)-Lf(x)|\ \to\ 0,\qquad N\to\infty.
   \end{equation}
   The involved operators are linear. It hence suffices to verify (\ref{uniconv}) for every monomial $f=f_n$, $n\in\nz_0$, defined via $f_n(x):=x^n$ for all $x\in[0,1]$. In the following (\ref{uniconv}) is even verified for all monotone functions $f\in C^2([0,1])$. Applying the space-scaling $x=n/N$ to (\ref{gen_rep}) it follows for all $N\in\nz$ and $x\in E_N$ that
   \begin{eqnarray*}
      &   & \hspace{-15mm}L_N(\pi_Nf)(x)
      \ = \ dxN\bigg(f\bigg(x-\frac{1}{N}\bigg)-f(x)\bigg)
            + \int_{(0,1]}
            \frac{\me\big(f(\frac{K_{xN,1-t}}{N})\big)-f(x)}{t}\,\Lambda({\rm d}t)\\
      &   & +\ d(1-x)N\bigg(f\bigg(x+\frac{1}{N}\bigg)-f(x)\bigg)
            + \int_{(0,1]}
            \frac{\me\big(f(1-\frac{K_{(1-x)N,1-t}}{N})\big)-f(x)}{t}\,\Lambda({\rm d}t).
   \end{eqnarray*}
   Therefore, for all $N\in\nz$ and $x\in E_N$,
   \begin{eqnarray*}
      |L_N(\pi_Nf)(x)-Lf(x)|
      & \le & dx\bigg|N\bigg(f\bigg(x-\frac{1}{N}\bigg)-f(x)\bigg)+f'(x)\bigg|\\
      &   & + \bigg|\int_{(0,1]}\frac{\me\big(f(\frac{K_{xN,1-t}}{N})\big)-f\big(x(1-t)\big)}{t}\,\Lambda({\rm d}t)\bigg|\\
      &   & + d(1-x) \bigg|N\bigg(f\bigg(x+\frac{1}{N}\bigg)-f(x)\bigg)-f'(x)\bigg|\\
      &   & + \bigg|\int_{(0,1]}\frac{\me\big(f(1-\frac{K_{(1-x)N,1-t}}{N})\big)
      -f\big(1-(1-x)(1-t)\big)}{t}\,\Lambda({\rm d}t)\bigg|.
   \end{eqnarray*}
   Uniformly over $x\in E_N$, all four terms on the right-hand side above converge to $0$ as $N\to\infty$: For the first and the third term this uniform convergence to $0$ is readily verified. For the second and the fourth term the pointwise convergence to $0$ follows from $K_{n,p}/n\to p$ in distribution as $n\to\infty$ and by dominated convergence. Since $f$ is assumed to be monotone, the maps $x\mapsto\me(f(K_{\lfloor xN\rfloor,1-t}/N))$ and $x\mapsto\me(f(1-K_{\lfloor(1-x)N\rfloor,1-t}/N))$ are monotone in $x\in[0,1]$. Thus, the convergence holds even uniformly for all $x\in[0,1]$ by P\'olya's \cite{Polya1920} classical theorem on convergence of monotone sequences of functions.
\end{proof}
\begin{proof}[Proof of Lemma \ref{generator_monomial}]
   In this proof all functions inside integrals are understood to be continuously extended at $t=0$.
   For all $n\in\nz$ and $x\in[0,1]$,
   \begin{eqnarray*}
      L^Xf_n(x)
      & = & \int_{[0,1]}\big(x^n(1-t)^n+(x(1-t)+t)^n-2x^n\big)
            \,\frac{\Lambda({\rm d}t)}{t}\\
         & = & \int_{[0,1]}
            \bigg(
               x^n(1-t)^n + \sum_{k=0}^n \binom{n}{k}t^k\big(x(1-t)\big)^{n-k} - 2x^n
            \bigg)\,\frac{\Lambda({\rm d}t)}{t}\\
      & = & \int_{[0,1]}
            \bigg(
               2x^n(1-t)^n + \sum_{k=1}^n \binom{n}{k}t^k\big(x(1-t)\big)^{n-k} - 2x^n
            \bigg)\,\frac{\Lambda({\rm d}t)}{t}\\
      & = & 2x^n\int_{[0,1]}\big((1-t)^n-1\big)
            \,\frac{\Lambda({\rm d}t)}{t}
            + \sum_{k=1}^n \binom{n}{k}x^{n-k}
            \int_{[0,1]}t^k(1-t)^{n-k}\,\frac{\Lambda({\rm d}t)}{t}\\
      & = & -2\varphi(n)x^n + \sum_{k=1}^n \varphi(k,n)x^{n-k}
   \end{eqnarray*}
   with $\varphi(k,n)$ and $\varphi(n)$ as defined in the statement of the lemma.
\end{proof}
\begin{proof}[Proof of Proposition \ref{intrep}]
   We have $2\varphi_n\mu_n=\sum_{k=1}^n \mu_{n-k}\varphi(k,n)$, $n\in\nz$. Multiplying this equation with $x^n/n!$ and summing over all $n\in\nz$ yields
   \[
   2\sum_{n=1}^\infty \mu_n\frac{x^n}{n!}\varphi(n)
   \ =\ \sum_{n=1}^\infty\frac{x^n}{n!}\sum_{k=1}^n\mu_{n-k}\varphi(k,n).
   \]
   The sum on the left-hand side is equal to
   \begin{eqnarray*}
      &   & \hspace{-15mm}\sum_{n=1}^\infty \mu_n\frac{x^n}{n!}
      \bigg(nd+\int_{(0,1]}\frac{1-(1-t)^n}{t}\,\Lambda({\rm d}t)\bigg)\\
      & = & d\sum_{n=1}^\infty \mu_n\frac{x^n}{(n-1)!}
            + \int_{(0,1]} \frac{1}{t}\sum_{n=1}^\infty \mu_n \big(1-(1-t)^n\big)\frac{x^n}{n!}\,\Lambda({\rm d}t)   \\
      & = & dx\frac{{\rm d}}{{\rm d}x}\sum_{n=1}^\infty \mu_n\frac{x^n}{n!}
            + \int_{(0,1]}\frac{1}{t}\sum_{n=1}^\infty \mu_n\frac{x^n-(x(1-t))^n}{n!}\,\Lambda({\rm d}t)
            \\
      & = & dxg'(x)+\int_{(0,1]}\frac{g(x)-g(x(1-t))}{t}\,\Lambda({\rm d}t).
   \end{eqnarray*}
   The right-hand side is equal to
   \begin{eqnarray*}
      &   & \hspace{-15mm}\sum_{n=1}^\infty \frac{x^n}{n!}\sum_{k=1}^n
            \mu_{n-k}\binom{n}{k}
            \int_{[0,1]}t^{k-1}(1-t)^{n-k}\,\Lambda({\rm d}t)\\
      & = & \int_{[0,1]}\sum_{k=1}^\infty\frac{x^kt^{k-1}}{k!}
            \sum_{n=k}^\infty \mu_{n-k}\frac{(x(1-t))^{n-k}}{(n-k)!}\,\Lambda({\rm d}t)\\
      & = & dxg(x) + \int_{(0,1]}\frac{e^{xt}-1}{t}g(x(1-t))\,\Lambda({\rm d}t),
   \end{eqnarray*}
   which yields (\ref{inteqn}). Conversely, assume that (\ref{inteqn}) holds for all $x\in[0,1]$, where $g$ is the mgf of some probability measure $Q$ on $[0,1]$. Expanding all expressions below the integrals in terms of powers of $x$ we can express both integrals in terms of the moments of $Q$. This yields that the moments of $Q$ satisfy the recursion (\ref{momrec}). Thus, $Q=\mu$ and, hence, $g$ is the mgf of $\mu$.
\end{proof}
\begin{proof}[Proof of Proposition \ref{dualityspecial}]
   Let $x\in[0,1]$ and $n\in\nz_0$. By the linearity of $L^X$ and (\ref{monomial}),
   \[
   L^XD(.,n)(x)
   \ =\ (n+1)L^Xf_n(x)
   \ =\ -2(n+1)h_nx^n + (n+1)\sum_{k=1}^n\frac{x^{n-k}}{k},
   \]
   where $h_n:=\Psi(n+1)-\Psi(1)=\sum_{i=1}^n 1/i$ denotes the $n$th harmonic number. On the other hand, by the definition of the generator $L^Y$,
   \begin{eqnarray*}
      &   & \hspace{-10mm}L^YD(x,.)(n)
      \ = \ \sum_{k=1}^n
            \big((n-k+1)x^{n-k}-(n+1)x^n\big)
            \underbrace{\bigg(\frac{1}{k}+\frac{1}{n-k+1}\bigg)}_{=\frac{n+1}{k(n-k+1)}}\\
      & = & -(n+1)x^n\sum_{k=1}^n\bigg(\frac{1}{k}+\frac{1}{n-k+1}\bigg)
            + \sum_{k=1}^n
            (n-k+1)x^{n-k}\frac{n+1}{k(n-k+1)}\\
      & = & -2(n+1)h_nx^n + (n+1)\sum_{k=1}^n \frac{x^{n-k}}{k}.
   \end{eqnarray*}
   Thus, $L^XD(.,n)(x)=L^YD(x,.)(n)$ for all $x\in[0,1]$ and $n\in\nz_0$, showing that $X$ is dual to $Y$ with respect to $D$.
\end{proof}
\begin{proof}[Proof of Theorem \ref{main2}]
   The linearity of the generator $L^X$ and its action on the $n$th monomial $f_n$ (see Lemma \ref{generator_monomial}) yields
   \begin{eqnarray*}
      L^XD(.,n)(x)
      & = & \frac{L^Xf_n(x)}{\mu_n}
      \ = \ \frac{1}{\mu_n}\sum_{k=1}^n \varphi(k,n)x^{n-k}
            -2\varphi(n)\frac{x^n}{\mu_n}\\
      & = & \frac{1}{\mu_n}\sum_{k=1}^n\varphi(k,n)\mu_{n-k}D(x,n-k)
            -2\varphi(n)D(x,n).
   \end{eqnarray*}
   Thanks to the recursion (\ref{momrec}) for the moments $\mu_n$, this expression can be rewritten as
   \begin{eqnarray*}
      L^XD(.,n)(x)
      & = & \frac{1}{\mu_n}\sum_{k=1}^n\varphi(k,n)\mu_{n-k}D(x,n-k)
            -\frac{1}{\mu_n}\sum_{k=1}^n\varphi(k,n)\mu_{n-k}D(x,n)\\
      & = & \frac{1}{\mu_n}\sum_{k=1}^n \mu_{n-k}\varphi(k,n) \big(D(x,n-k)-D(x,n)\big)
      \ = \ L^Y D(x,.)(n),
   \end{eqnarray*}
   where $Y$ is a continuous-time death process with state space $\nz_0$ and generator (\ref{Ygen})
   Thus, $L^XD(.,n)(x)=L^YD(x,.)(n)$ for all $x\in[0,1]$ and $n\in\nz_0$,
   showing that $X$ is dual to $Y$ with respect to $D$.
\end{proof}
In the following, as in the proof of Theorem \ref{main1}, let $K_{n,p}$ denote a random variable having a binomial distribution with parameters $n\in\nz$ and $p\in[0,1]$. For the proof of Lemma \ref{generator_hypergeometric}, the following \emph{sampling duality} result for binomial distributions is needed.
\begin{lemma}[Sampling duality for the Binomial distribution] \label{aux}
   For all $N\in\nz$, $p\in[0,1]$ and $n,m\in\{0,\ldots,N\}$,
   \begin{equation}
      \me\bigg(\frac{(n+K_{N-n,p})_m}{(N)_m}\bigg)\ =\ \me\bigg(\frac{(n)_{m-K_{m,p}}}{(N)_{m-K_{m,p}}}\bigg).
   \end{equation}
\end{lemma}
\begin{proof}
   Consider a population of size $N$ made of $n$ individuals of type $A$ and $N-n$ individuals of type $B$. Reproduction works as follows: each $A$ always survives (equivalently, it produces exactly one descendant); each $B$ survives independently with probability $p$ (equivalently, it produces one offspring with probability $p$ and has no offspring with probability $1-p$). Thus, the number of $B$’s that survive is described by the binomial random variable $K_{N-n,p}$. After the reproduction step, the next generation contains $n+K_{N-n,p}$ individuals.

   Consider sampling $m$ individuals from the original population of size $N$. The result follows by computing the probability of the event
   $E:=\{$all sampled individuals have descendants after the reproduction step$\}$ in the following two ways.
   \begin{itemize}
      \item\emph{Thinning followed by sampling:} after thinning, $A$’s all
         survive, while exactly $K_{N-n,p}$ $B$’s survive. So there are $n+K_{N-n,p}$ individuals whose line survives. Given $K_{N-n,p}$, the probability that a sample of $m$ individuals from the $N$ original individuals all belong to this surviving set is
         $(n+K_{N-n,p})_m/(N)_m$. Taking expectation yields
         \begin{equation}
            \mathbb{P}(E)\ =\
            \mathbb{E}\left(\frac{(n+K_{N-n,p})_m}{(N)_m}\right).
         \end{equation}
      \item\emph{Sampling followed by thinning:} we now reverse the order
         and first sample $m$ individuals. Let $Z$ denote the number of $B$’s in
         the sample. Then $Z$ has a hypergeometric distribution with parameters $N$, $N-n$ and $m$. Among these $Z$ sampled $B$'s, each survives independently with probability $p$. Let $Y$ be the number of sampled $B$’s that survive. Conditional on $Z$, we have $Y|Z\stackrel{d}{=}\text{B}(Z,p)$. A standard property of binomial thinning gives
        $Y\stackrel{d}{=}K_{m,p}$, i.e. the hypergeometric–binomial mixture collapses to a binomial distribution. For all sampled $m$ individuals to survive, the remaining $m-Y$ individuals must correspond to $A$ individuals. The probability that the $m-K_{m,p}$ surviving $A$'s come from the $n$ $A$'s among the $N$ individuals is $(n)_{m-K_{m,p}}/(N)_{m-K_{m,p}}$. Taking expectation yields
        \begin{equation}
           \mathbb{P}(E)\ =\ \me\left(\frac{(n)_{m -K_{m,p}}}{(N)_{m-K_{m,p}}}\right).
        \end{equation}
   \end{itemize}
    The proof is complete.
\end{proof}
\begin{remark} (Sampling duality).
   Lemma \ref{aux} can be stated in terms of Markov chains as follows. Let $B_j^{(n)}$, $j\in\nz,n\in\nz_0$, be iid Bernoulli random variables with parameter $p\in[0,1]$. Fix $N\in\nz$. Given two initial variables $X_0$ and $Y_0$ taking values in $\{0,\ldots,N\}$, define two Markov chains $X:=(X_n)_{n\in\nz_0}$ and $Y:=(Y_n)_{n\in\nz_0}$ recursively via
   \[
   X_{n+1}\ :=\ X_n+\sum_{j=1}^{N-X_n}B_j^{(n)}
   \qquad\mbox{and}\qquad
   Y_{n+1}\ :=\ Y_n-\sum_{j=1}^{Y_n}B_j^{(n)},\qquad n\in\nz_0.
   \]
   Lemma \ref{aux} implies that $X$ is dual to $Y$ with respect to the sampling duality kernel $H_N:\{0,\ldots,N\}^2\to[0,1]$ defined via $H_N(n,m):=(n)_m/(N)_m$, $n,m\in\{0,\ldots,N\}$.
\color{black}
   Similar duality result are known for forward and backward processes in exchangeable Cannings population models
   (see, for example, \cite[Theorem 3.1]{Moehle1999} and can be traced back to works of Cannings \cite{Cannings1974} and Gladstien \cite{Gladstien1976,Gladstien1977,Gladstien1978}, where sampling duality results are stated in elementary matrix notation without knowing the notion of duality. For a more recent related work using sampling duality techniques we refer the reader to
   Gonz\'alez Casanova and Span\`o \cite{GonzalezCasanovaSpano2018}.
\color{blue}
\end{remark}
\color{black}
\begin{proof}[Proof of Lemma \ref{generator_hypergeometric}]
   We have
   \[
   L^{X^{(N)}}h_m^{(N)}(n)
      \ =\ \sum_{k=1}^n\big(h_m^{(N)}(n-k)-h_m^{(N)}(n)\big)\varphi(k,n)
            +\sum_{k=1}^{N-n}\big(h_m^{(N)}(n+k)-h_m^{(N)}(n)\big)\varphi(k,N-n).
   \]
   The proof of Lemma \ref{generator_hypergeometric} can hence be split into the proof of
   \begin{equation} \label{first}
      \sum_{k=1}^n\big(h_m^{(N)}(n-k)-h_m^{(N)}(n)\big)\varphi(k,n)
      \ =\ -\varphi(m)h_m^{(N)}(n)
   \end{equation}
   and of
   \begin{equation} \label{second}
      \sum_{k=1}^{N-n}\big(h_m^{(N)}(n+k)-h_m^{(N)}(n)\big)\varphi(k,N-n)
      \ =\ \sum_{k=1}^m\varphi(k,m)h_{m-k}^{(N)}(n)-\varphi(m)h_m^{(N)}(n).
   \end{equation}
   As in the proof of Theorem \ref{main1}, let $K_{n,p}$ denote a random variable having a binomial distribution with parameters $n\in\nz$ and $p\in[0,1]$. Note that $K_{n,p}$ has descending factorial moments $\me((K_{n,p})_m)=p^m(n)_m$, $m\in\nz_0$. Division by $(N)_m$ shows that $\me(h_m^{(N)}(K_{n,p}))=p^mh_m^{(N)}(n)$, $m\in\nz_0$. Note furthermore that $n-K_{n,p}$ has the same distribution as $K_{n,1-p}$. Thus,
   \begin{equation} \label{binmom}
      \me\big(h_m^{(N)}(n-K_{n,p})\big)
      \ =\ \me\big(h_m^{(N)}(K_{n,1-p})\big)\ =\ (1-p)^mh_m^{(N)}(n),
      \qquad m\in\nz_0.
   \end{equation}
   In the following all functions inside integrals are understood to be continuously extended at $t=0$.
   Using the definition of $\varphi(k,n)$ and (\ref{binmom}), the left hand side of (\ref{first}) simplifies to
   \begin{eqnarray*}
      &   & \hspace{-15mm}\sum_{k=1}^n
            \big(h_m^{(N)}(n-k)-h_m^{(N)}(n)\big)\varphi(k,n)\\
      & = & \sum_{k=1}^n\big(h_m^{(N)}(n-k)-h_m^{(N)}(n)\big)\int_{[0,1]}
            \pr(K_{n,t}=k)\frac{\Lambda({\rm d}t)}{t}\\
      & = & \int_{[0,1]}\sum_{k=0}^n \big(h_m^{(N)}(n-k)-h_m^{(N)}(n)\big)
            \pr(K_{n,t}=k)\frac{\Lambda({\rm d}t)}{t}\\
      & = & \int_{[0,1]}\Big(\me\big(h_m^{(N)}(n-K_{n,t})\big)-h_m^{(N)}(n)\Big)
            \frac{\Lambda({\rm d}t)}{t}\\
      & = & \int_{[0,1]}\big((1-t)^m h_m^{(N)}(n)-h_m^{(N)}(n)\big)
            \frac{\Lambda({\rm d}t)}{t}\\
      & = & h_m^{(N)}(n)\int_{[0,1]}\frac{(1-t)^m-1}{t}\,\Lambda({\rm d}t)
      \ = \ -\varphi(m)h_m^{(N)}(n),
   \end{eqnarray*}
   where the last equality holds by the integral formula for the total rate $\varphi(m)$. Thus, (\ref{first}) is established. It remains to verify (\ref{second}). Using the definition of $\varphi(k,n)$, the left hand side of (\ref{second}) turns into
   \begin{eqnarray*}
      &   & \hspace{-15mm}
            \sum_{k=1}^{N-n}
            \big(h_m^{(N)}(n+k)-h_m^{(N)}(n)\big)\varphi(k,N-n)\\
      & = & \sum_{k=1}^{N-n}
            \big(h_m^{(N)}(n+k)-h_m^{(N)}(n)\big)\int_{[0,1]}\pr(K_{N-n,t}=k)
            \frac{\Lambda({\rm d}t)}{t}\\
      & = & \int_{[0,1]}
            \sum_{k=0}^{N-n}\big(h_m^{(N)}(n+k)-h_m^{(N)}(n)\big)\pr(K_{N-n,t}=k)
            \frac{\Lambda({\rm d}t)}{t}\\
      & = & \int_{[0,1]}
            \Big(
               \me\big(h_m^{(N)}(n+K_{N-n,t})\big)-h_m^{(N)}(n)
            \Big)\frac{\Lambda({\rm d}t)}{t}
   \end{eqnarray*}
   Similarly, for the right hand side of (\ref{second}) we obtain
   \[
   \sum_{k=1}^m \varphi(k,m)h_{m-k}^{(N)}(n) - \varphi(m)h_m^{(N)}(n)
   \ =\ \int_{[0,1]} \Big(
           \me\big(h_{m-K_{m,t}}^{(N)}(n)\big) - h_m^{(N)}(n)
        \Big)\frac{\Lambda({\rm d}t)}{t}.
   \]
   Thus, (\ref{second}) holds, since $\me\big(h_m^{(N)}(n+K_{N-n,t})\big)=\me\big(h_{m-K_{m,t}}^{(N)}(n)\big)$ by Lemma \ref{aux}.
\end{proof}
\begin{proof}[Proof of Theorem \ref{main3}]
   We essentially follow the proof of Theorem \ref{main2}. The calculations are slightly more involved due to the discrete nature of the involved processes, requiring a replacement of monomials by factorials.
   The linearity of the generator $L^{X^{(N)}}$ and its action on the function $h_m^{(N)}$ (see Lemma \ref{generator_hypergeometric}) yields
   \begin{eqnarray*}
      L^{X^{(N)}}D^{(N)}(.,m)(n)
      & = & \frac{L^{X^{(N)}}h_m^{(N)}(n)}{\mu_m^{(N)}}
      \ = \ \frac{1}{\mu_m^{(N)}}\sum_{k=1}^m \varphi(k,m)h_{m-k}^{(N)}(n)
            -2\varphi(m)\frac{h_m^{(N)}(n)}{\mu_m^{(N)}}\\
      & = & \frac{1}{\mu_m^{(N)}}\sum_{k=1}^m\varphi(k,m)\mu_{m-k}^{(N)}D^{(N)}(n,m-k)
            -2\varphi(m)D^{(N)}(n,m).
   \end{eqnarray*}
   Thanks to the recursion (\ref{recursionhyper}) for the $\mu_m^{(N)}$, this expression can be rewritten as
   \begin{eqnarray*}
      &   & \hspace{-15mm}L^{X^{(N)}}D^{(N)}(.,m)(n)\\
      & = & \frac{1}{\mu_m^{(N)}}\sum_{k=1}^m\varphi(k,m)\mu_{m-k}^{(N)}D^{(N)}(n,m-k)
            -\frac{1}{\mu_m^{(N)}}\sum_{k=1}^m\varphi(k,m)\mu_{m-k}^{(N)}D^{(N)}(n,m)\\
      & = & \frac{1}{\mu_m^{(N)}}\sum_{k=1}^m \mu_{m-k}^{(N)}\varphi(k,m) \big(D^{(N)}(n,m-k)-D^{(N)}(n,m)\big)
      \ = \ L^{Y^{(N)}} D^{(N)}(n,.)(m),
   \end{eqnarray*}
   where $Y^{(N)}$ is a continuous-time death process with state space $\{0,\ldots,N\}$ and generator (\ref{YNgen}).
   Thus, $L^{X^{(N)}}D^{(N)}(.,m)(n)=L^{Y^{(N)}}D^{(N)}(n,.)(m)$ for all $n,m\in\{0,\ldots,N\}$
   showing that $X^{(N)}$ is dual to $Y^{(N)}$ with respect to $D^{(N)}$.
\end{proof}
\begin{proof}[Proof of Theorem \ref{main4}]
   For any function $f:\nz_0\to\rz$ let the function $\pi_Nf:\{0,\ldots,N\}\to\rz$ be defined by $\pi_Nf(m):=f(m)$ for all $m\in\{0,\ldots,N\}$.
   To prove the convergence $Y^{(N)}\to Y$ in $D_{\nz_0}([0,\infty))$, it suffices by standard theory on weak convergence of Markov processes (see, for example, Ethier and Kurtz \cite[p.~28, Theorem 6.1 and p.~232, Corollary 8.7]{EthierKurtz2005}) to verify that, for every bounded function $f:\nz_0\to\rz$,
   \begin{equation} \label{normconv}
      \|L^{Y^{(N)}}(\pi_Nf)-\pi_N(L^Yf)\|
      \ :=\ \sup_{m\in\{0,\ldots,N\}}|L^{Y^{(N)}}(\pi_Nf)(m)-L^Yf(m)|
      \ \to\ 0,\quad N\to\infty.
   \end{equation}
   In the following it is shown that the left hand side in (\ref{normconv}) is even equal to $0$ for all $N\in\nz$. As observed after Eq.~(\ref{recursionhyper}), the scaled factorial moment $\mu_m^{(N)}$ coincides with the moment $\mu_m$ as long as $m\le N$. Hence, for all $N\in\nz$ and all $m\in\{0,\ldots,N\}$, by (\ref{YNgen}) and (\ref{Ygen}),
   \[
   L^{Y^{(N)}}(\pi_Nf)(m)
   \ =\ \frac{1}{\mu_m}\sum_{k=1}^m \varphi(k,m)\mu_{m-k}\big(f(m-k)-f(m)\big)
   \ =\ L^Yf(m)
   \]
   does not depend on $N\in\nz$. Thus, the equality $L^{Y^{(N)}}(\pi_Nf)=\pi_N(L^Yf)$ holds for all $N\in\nz$, showing that the left hand side in (\ref{normconv}) is equal to $0$ for all $N\in\nz$.
\color{black}
\end{proof}
\begin{proof}[Proof of Theorem \ref{duality}]
   Let $D$ be a kernel of the polynomial form (\ref{polynomial}).
   Then, from the linearity of the generator $L^X$ and from (\ref{monomial}) it follows that
   \begin{eqnarray*}
      L^XD(.,n)(x)
      & = & \sum_{j=0}^n d_{n,j} L^X f_j(x)
      \ = \ \sum_{j=0}^n d_{n,j}\bigg(
            -2\varphi(j)x^j + \sum_{k=1}^j \varphi(k,j)x^{j-k}\bigg)\\
      & = & -2\sum_{j=0}^n d_{n,j}\varphi(j)x^j
            +\sum_{j=0}^n d_{n,j}\sum_{i=0}^{j-1}\varphi(j-i,j) x^i\\
      & = & -2\sum_{i=0}^n d_{n,i}\varphi(i)x^i
            + \sum_{i=0}^{n-1}\sum_{j=i+1}^n d_{n,j}\varphi(j-i,j) x^i.
   \end{eqnarray*}
   This expression is a polynomial of degree less than or equal to $n$ and the coefficient in front of $x^i$, $i\in\{0,\ldots,n\}$, is
   \begin{equation} \label{coef1}
      -2d_{n,i}\varphi(i) + \sum_{j=i+1}^n d_{n,j}\varphi(j-i,j),
      \qquad i\in\{0,\ldots,n\}.
   \end{equation}
   On the other hand, by the definition (\ref{generatory}) of the generator $L^Y$,
   \begin{eqnarray*}
      &   & \hspace{-15mm} L^YD(x,.)(n)
      \ = \ \sum_{k=1}^n
            \big(\varphi(k,n)+\varphi(n-k+1,n)\big)\big(D(x,n-k)-D(x,n)\big)\\
      & = & \sum_{k=1}^n\big(\varphi(k,n)+\varphi(n-k+1,n)\big)
            \bigg(
            \sum_{j=0}^{n-k}d_{n-k,j}x^j-\sum_{j=0}^n d_{n,j}x^j
            \bigg)\\
      & = & \sum_{j=0}^{n-1} x^j\sum_{k=1}^{n-j}
            d_{n-k,j}\big(\varphi(k,n)+\varphi(n-k+1,n)\big)\\
      &   & \hspace{1cm} - \sum_{j=0}^n d_{n,j}x^j
            \underbrace{\sum_{k=1}^n\big(\varphi(k,n)+\varphi(n-k+1,n)\big)}_{=2\varphi(n)}.
   \end{eqnarray*}
   Again, this expression is a polynomial of degree less than or equal to $n$ and the coefficient in front of $x^i$, $i\in\{0,\ldots,n\}$, is
   \begin{equation} \label{coef2}
      \sum_{k=1}^{n-i}d_{n-k,i}\big(\varphi(k,n)+\varphi(n-k+1,n)\big)
       - 2d_{n,i}\varphi(n).
   \end{equation}
   We would like to choose the coefficients $d_{n,i}$ in such a way that
   (\ref{coef1}) and (\ref{coef2}) coincide, since then we will have
   the equality $L^XD(.,n)(x)=L^YD(x,.)(n)$, which is the desired duality result. Putting (\ref{coef1}) equal to (\ref{coef2}) yields the
   equation
   \begin{equation}\label{recursion}
      2d_{n,i}\varphi(i)
      \ =\ \sum_{k=1}^{n-i}d_{n-k,i}\big(\varphi(k,n)+\varphi(n-k+1,n)\big)
           -\sum_{j=i+1}^n d_{n,j}\varphi(j-i,j).
   \end{equation}
   Now note that (\ref{recursion}) always holds for $i=n$ no matter how $d_{n,n}$ is chosen. Thus, one can choose the diagonal coefficients $d_{n,n}$, $n\in\nz_0$, arbitrarily. Once these diagonal coefficients are chosen, say as $d_{n,n}:=r_n$ for all $n\in\nz_0$, Eq.~(\ref{recursion}) is a recursion forwards over $n$ and for each $n$ backwards over $i=n-1,n-2,\ldots,1,0$, showing that all the off-diagonal coefficients $d_{n,i}$, $n\in\nz$, $i\in\{0,\ldots,n-1\}$, are uniquely determined by the diagonal sequence $r=(r_n)_{n\in\nz_0}$.
\end{proof}
\begin{proof}[Proof of Proposition \ref{betastat}]
   (i) ``$\Rightarrow$'': Assume first that $\mu=\beta(p,q)$, i.e. the stationary distribution of the $\beta(a,b)$-mutation process $X$ is a beta distribution for some parameters $p,q>0$. We have to verify that $a\in\{1,2\}$ or $a=b$. From the recursion (\ref{momrec}) it follows that the first four moments of $\mu$ are
   \begin{equation} \label{firstfourmoments}
      \begin{array}{lll}
      \mu_1 & = & \displaystyle\frac{1}{2},\quad \mu_2\ =\ \frac{a+b}{2(a+2b)},\quad
      \mu_3\ =\ \displaystyle\frac{2a+b}{4(a+2b)}\quad\mbox{and}\quad\\
      & & \\
      \mu_4 & = & \displaystyle\frac{(a+b+2)(a^3+a^2+3a^2b+4ab+4ab^2+b^2+b^3) }{2(a+2b)(a+2b+2)(a^2+2ab+a+4b+2b^2)}.
      \end{array}
   \end{equation}
   Note that the formula for $\mu_4$ is a rather involved expression in $a$ and $b$. Comparing $\mu_1=\frac{1}{2}$ with the mean $\frac{p}{p+q}$ of the $\beta(p,q)$-distribution it follows that $p=q$. The symmetric beta distribution $\beta(p,p)$ has second moment $\frac{p(p+1)}{2p(2p+1)}=\frac{p+1}{2(2p+1)}$. Since this expression is equal to $\mu_2$ it follows that $p=b/a$.  For this $p$, even the third moment $\frac{p(p+1)(p+2)}{2p(2p+1)(2p+2)}=\frac{p+2}{4(2p+1)}=\frac{2a+b}{4(a+2b)}$ of the $\beta(p,p)$-distribution coincides with $\mu_3$. However, the forth moment $\frac{p(p+1)(p+2)(p+3)}{2p(2p+1)(2p+2)(2p+3)}=
   \frac{(p+2)(p+3)}{4(2p+1)(2p+3)}=
   \frac{(2a+b)(3a+b)}{4(a+2b)(3a+2b)}$
   only coincides with $\mu_4$ in particular cases. A technical but straightforward calculation shows that these cases are
   $a\in\{1,2\}$ or $a=b$.

   ``$\Leftarrow$'': We have to show that, for $a\in\{1,2\}$ or $a=b$, the stationary distribution is the $\beta(b/a,b/a)$-distribution. To see this it suffices to verify that the moments of the beta distribution $\mu:=\beta(b/a,b/a)$ satisfies the recursion (\ref{momrec}). We do this separately for the three cases $a=1$, $a=2$ or $a=b$.

   \vspace{2mm}

   \textbf{Case 1}: Assume that $a=1$.
   In the following, instead of verify the recursion (\ref{momrec}), it is even shown that $\mu:=\beta(b,b)$ is reversible, i.e.~for all functions $f,g\in L_2(\mu)$ it is shown that
   \begin{equation} \label{adjoint}
      \langle L^Xf,g\rangle\ =\ \langle f,L^Xg\rangle,
   \end{equation}
   where $\langle\cdot,\cdot\rangle$ denotes the scalar product in $L_2(\mu)$. Recalling the definition of $L^X$ we have $\langle L^Xf,g\rangle=I_1+I_2$, where
   \[
   I_1\ :=\ \int_0^1 \left(\int_0^1 \big(f(x(1-t))-f(x)\big) \frac{1}{t}(1-t)^{b-1}\,{\rm d}t\right)g(x)\frac{x^{b-1}(1-x)^{b-1}}{{\rm B}(b,b)}\,{\rm d}x
   \]
   and
   \[
   I_2\ :=\ \int_0^1\left(\int_0^1\big(f(x(1-t)+t)-f(x)\big)\frac{1}{t}(1-t)^{b-1}
   \,{\rm d}t\right)g(x)\frac{x^{b-1}(1-x)^{b-1}}{{\rm B}(b,b)}\,{\rm d}x.
   \]
   The change of variables $(x,t)\mapsto(x,y)$ with $y:=x(1-t)$ yields
   \[
   I_1\ =\ \frac{1}{{\rm B}(b,b)}\iint\limits_{0<y<x<1}\big(f(y)-f(x)\big)g(x) \frac{y^{b-1}(1-x)^{b-1}}{x-y}\,{\rm d}x\,{\rm d}y
   \]
   and changing variables $(x,t)\mapsto(x,y)$ with $y:=x(1-t)+t$ gives
   \[
   I_2\ =\ \frac{1}{{\rm B}(b,b)}\iint\limits_{0<x<y<1}\big(f(y)-f(x)\big)g(x) \frac{(1-y)^{b-1}x^{b-1}}{y-x}\,{\rm d}x\,{\rm d}y.
   \]
   In a similar manner, one shows that $\langle f,L^Xg\rangle=I_2+I_1$.
   As a consequence \eqref{adjoint} holds.

   \color{black}

   \vspace{2mm}

   \textbf{Case 2}: Assume that $a=2$. The beta distribution $\beta(\frac{b}{2},\frac{b}{2})$ has moments $m_n:=\frac{\Gamma(b)\Gamma(n+\frac{b}{2})}{\Gamma(\frac{b}{2})\Gamma(n+b)}$. Therefore, using \eqref{rate},
   \begin{eqnarray*}
      \sum_{k=1}^n m_{n-k}\varphi_{2,b}(k,n)
      & = & \sum_{k=1}^n
            \frac{\Gamma(b)\Gamma(n-k+\frac{b}{2})}
                 {\Gamma(\frac{b}{2})\Gamma(n-k+b)}
            \binom{n}{k}\frac{1}{{\rm B}(2,b)}\frac{\Gamma(k+1)\Gamma(n-k+b)}{\Gamma(n+b+1)}\\
      & = & \frac{1}{{\rm B}(2,b)}\frac{\Gamma(b)}{\Gamma(\frac{b}{2})}
            \frac{n!}{\Gamma(n+b+1)}
            \sum_{k=1}^n\frac{\Gamma(n-k+\frac{b}{2})}{(n-k)!}.
   \end{eqnarray*}
   The last sum simplifies to $\sum_{i=0}^{n-1}\frac{\Gamma(i+\frac{b}{2})}{i!}=\frac{\Gamma(n+\frac{b}{2})}{\frac{b}{2}\Gamma(n)}$, which is easily verified by induction on $n\in\nz$. Thus,
   \begin{eqnarray*}
      \sum_{k=1}^n m_{n-k}\varphi_{2,b}(k,n)
      & = & \frac{1}{{\rm B}(2,b)}\frac{\Gamma(b)}{\Gamma(\frac{b}{2})}\frac{n!}{\Gamma(n+b+1)}
   \frac{\Gamma(n+\frac{b}{2})}{\frac{b}{2}\Gamma(n)}\\
      & = & \frac{1}{{\rm B}(2,b)}2\frac{\Gamma(b)\Gamma(n+\frac{b}{2})}{\Gamma(\frac{b}{2})\Gamma(n+b)}
    \frac{n}{b(n+b)}
   =2m_n\varphi_{2,b}(n),
   \end{eqnarray*}
   since $\varphi_{2,b}(n)=\frac{1}{{\rm B}(2,b)}\frac{n}{b(n+b)}$ by \eqref{total-rate-ab}. Thus, the
   moments $m_n$ satisfy the recursion (\ref{momrec}).

   \vspace{2mm}

   \textbf{Case 3}: For the case $a=b>0$, the moments $m_n:=1/(n+1)$ of the uniform distribution satisfy
   \begin{eqnarray*}
      \sum_{k=1}^n m_{n-k}\varphi_{a,a}(k,n)
      & = & \sum_{k=1}^n \frac{1}{n-k+1}\binom{n}{k}
            \frac{1}{{\rm B}(a,a)}\int_0^1 t^{k+a-2}(1-t)^{n-k+a-1}\,{\rm d}t\\
      & = & \frac{1}{n+1}\frac{1}{{\rm B}(a,a)}\int_0^1 t^{a-2}(1-t)^{a-2}
            \sum_{k=1}^n\binom{n+1}{k}t^k(1-t)^{n+1-k}\,{\rm d}t\\
      & = & m_n \frac{1}{{\rm B}(a,a)}\int_0^1 t^{a-2}(1-t)^{a-2}
            \big(1-(1-t)^{n+1}-t^{n+1}\big)\,{\rm d}t.
   \end{eqnarray*}
   For $a\ne 1$, the last integral is equal to
   ${\rm B}(a-1,a-1)-{\rm B}(a-1,n+a)-{\rm B}(n+a,a-1)
   =2{\rm B}(a-1,a) - 2 {\rm B}(a-1,n+a)
   =2{\rm B}(a,a)\varphi_{a,a}(n)$,
   since ${\rm B}(a-1,a-1)=2{\rm B}(a-1,a)$, and for $a=1$ this integral is equal to
   \begin{eqnarray*}
      \int_0^1 \frac{1-(1-t)^{n+1}-t^{n+1}}{t(1-t)}\,{\rm d}t
      & = & \int_0^1 \bigg(\frac{1-t^n}{1-t} + \frac{1-(1-t)^n}{t}\bigg)\,{\rm d}t\\
      & = & 2\big(\Psi(n+1)-\Psi(1)\big)
      \ = \ 2\varphi_{1,1}(n).
   \end{eqnarray*}
   Thus, again, the recursion (\ref{momrec}) holds,
   which shows that the uniform distribution is the stationary distribution for the case $a=b$.

   (ii) If $a=1$ then it was already verified in (i) that the detailed balance condition holds. Thus, for $a=1$ the stationary distribution is reversible.

   Conversely, assume that the stationary distribution \textcolor{blue}{$\mu$} is reversible.
   We have to verify that $a=1$. Since $\mu$ is reversible, we have $\langle L^Xf_n,f_m\rangle=\langle f_n,L^Xf_m\rangle$ for all $n,m\in\nz_0$. Using the action (\ref{monomial}) of $L^X$ on monomials it is readily seen that these equations are equivalent to
   \begin{equation} \label{reversibility}
      \sum_{k=1}^n \varphi(k,n)\mu_{n+m-k} - 2\varphi(n)\mu_{n+m}
      \ =\ \sum_{k=1}^m \varphi(k,m)\mu_{n+m-k} - 2\varphi(m)\mu_{n+m},
   \quad n,m\in\nz_0,
   \end{equation}
   where $\mu_n:=\int f_n\,{\rm d}\mu$ denotes the $n$th moment of $\mu$, $n\in\nz_0$. For $m=0$, the right hand side in (\ref{reversibility}) vanishes, leading to the recursion (\ref{momrec}) for the moments of $\mu$. In particular, the first four moments of $\mu$ are given by (\ref{firstfourmoments}). The obvious idea of obtaining the desired result $a=1$ by exploiting (\ref{reversibility}) for $(m,n)=(1,2)$ fails, since it turns out that (\ref{reversibility}) is satisfied for all $a,b>0$ in this case. The next attempt is hence to consider $(m,n)=(1,3)$, in which case (\ref{reversibility}) turns into
   \begin{eqnarray*}
      &   & \hspace{-10mm}0
      \ = \ \big(\varphi(1,3)-\varphi(1,1)\big)\mu_3 + \varphi(2,3)\mu_2 + \varphi(3,3)\mu_1
      +2\big(\varphi(1)-\varphi(3)\big)\mu_4\\
      & = & \bigg(\frac{3b(b+1)}{(a+b)(a+b+1)}-1\bigg)\frac{2a+b}{4(a+2b)}\\
      &   & \hspace{1mm}+\frac{3ab}{(a+b)(a+b+1)}\frac{a+b}{2(a+2b)}
            +\frac{a(a+1)}{(a+b)(a+b+1)}\frac{1}{2}\\
      &   & \hspace{1mm}
            + 2\bigg(1-\frac{a^2+3ab+a+3b^2+3b}{(a+b)(a+b+1)}\bigg)
            \frac{(a+b+2)(a^3+a^2+3a^2b+4ab+4ab^2+b^2+b^3) }{2(a+2b)(a+2b+2)(a^2+2ab+a+4b+2b^2)}
            \\
      & = & \frac{ab(1-a^2)}{4(a^2+2ab+a+4b+2b^2)(a+b)(a+b+1)},
   \end{eqnarray*}
   where the last equality follows via tedious but straightforward calculus.
   The latter fraction is equal to $0$ only for $a=1$. The proof is complete.
\color{black}
\end{proof}
\begin{proof}[Proof of Lemma \ref{stieltjes}]
   As in Handa \cite{Handa2014}, define $G_t(x):=1/(1+tx)$. A straightforward calculation shows that
   \[
   G_t(x(1-u))+G_t(x(1-u)+u)-2G_t(x)
   \ =\ -\frac{tu}{1+tx}\bigg(\frac{1-x}{1+tx(1-u)+tu}-\frac{x}{1+tx(1-u)}\bigg).
   \]
   Therefore,
   \begin{eqnarray*}
      L^XG_t(x)
      & = & \int_0^1 \frac{u^{a-1}(1-u)^{b-1}}{u}
            \big(G_t(x(1-u))+G_t(x(1-u)+u)-2G_t(x)\big)\,{\rm d}u\\
      & = & -\int_0^1 u^{a-1}(1-u)^{b-1}\frac{t}{1+tx}
      \bigg(\frac{1-x}{1+tx(1-u)+tu}-\frac{x}{1+tx(1-u)}\bigg)\,{\rm d}u\\
      & = & -\frac{t(1-x)}{1+tx}\int_0^1 \frac{u^{a-1}(1-u)^{b-1}}{1+tx(1-u)+tu}\,{\rm d}u
      +\frac{tx}{1+tx}\int_0^1 \frac{u^{a-1}(1-u)^{b-1}}{1+tx(1-u)}\,{\rm d}u.
   \end{eqnarray*}
   Since $a+b=1$, both integrals are known explicitly (Handa \cite[Eq.~(2.4)]{Handa2014}), which yields
   \begin{eqnarray*}
      L^XG_t(x)
      & = & -\frac{t(1-x)}{1+tx}{\rm B}(a,b)(1+t)^{-a}(1+tx)^{-b}
   +\frac{tx}{1+tx}{\rm B}(a,b)(1+tx)^{-b}\\
      & = & -{\rm B}(a,b)t\bigg(\frac{1-x}{(1+t)^{1-b}(1+tx)^{b+1}}
   -\frac{x}{(1+tx)^{b+1}}\bigg).
   \end{eqnarray*}
   For all $t\ge 0$, $\int L^XG_t(x)\,\mu({\rm d}x)=0$, which gives
   \[
   \int\frac{x}{(1+tx)^{b+1}}\,\mu({\rm d}x)
   \ =\ \frac{1}{(1+t)^{1-b}}\int\frac{1-x}{(1+tx)^{b+1}}\,\mu({\rm d}x).
   \]
   The left and right integrals are equal to $-S_b'(t)/b$ and
   $\frac{1+t}{b}S_b'(t)+S_b(t)$ respectively, which yields
   the differential equation $\big((1+t)^{1-b}+1+t\big)S_b'(t)=-bS_b(t)$
   and, hence, via integration
   \[
   -\log S_b(t)\ =\ \int_0^t \frac{b}{(1+u)^{1-b}+1+u}\,{\rm d}u,
   \qquad t\ge0.
   \]
   An antiderivative function is $F(u):=\log(1+(1+u)^b)$, which yields
   $-\log S_b(t)
   =F(t)-F(0)
   =\log(1+(1+t)^b)-\log 2
   =\log\big((1+(1+t)^b)/2\big)$
   and, hence, the solution (\ref{S}).
\end{proof}
\begin{proof}[Proof of Lemma \ref{aplusbequal1}]
   The substitution $y:=1/x$ yields
   \[
   S_b(z)
   \ =\ \int_0^1(1+zx)^{-b}f_\mu(x)\,{\rm d}x\\
   \ =\ \int_1^\infty(1+z/y)^{-b} f_\mu(1/y)y^{-2}\,{\rm d}y\\
   \ =\ \int_1^\infty (y+z)^{-b}F(y)\,{\rm d}y,
   \]
   where $F(y):=y^{b-2}f_\mu(1/y)$ for all $y>1$. By inversion (see, for example, Schwarz \cite[Eq.~(9)]{Schwarz2005}),
   \[
   F(y)\ :=\ -\frac{1}{2\pi i}y^b\int_C (1+w)^{b-1}S_b'(yw)\,{\rm d}w,\qquad y\in(1,\infty),
   \]
   where $C$ is a contour that starts and ends at the point $w=-1$ and encloses the origin. Multiplication with $y^{2-b}$ yields
   \[
   f_\mu(1/y)\ =\ y^{2-b}F(y)
   \ =\ -\frac{1}{2\pi i}y^2\int_C (1+w)^{b-1}S_b'(yw)\,{\rm d}w,\qquad y\in(1,\infty),
   \]
   and back substitution $y:=1/x$ leads to
   \[
   f_\mu(x)\ =\ -\frac{1}{2\pi i}x^{-2}\int_C (1+w)^{\textcolor{blue}{b}-1}S_b'(w/x)\,{\rm d}w,
   \qquad x\in(0,1).
   \]
   The result now follows from $S_b'(z)=-\frac{2b(1+z)^{b-1}}{(1+(1+z)^b)^2}$.
\end{proof}
\begin{proof}[Proof of Proposition \ref{specialmoments}]
   The expansion of the generalized Stieltjes transform $S_b(t)$ in powers of $t$ is obtained via
   \begin{eqnarray*}
      S_b(t)
      & = & \frac{2}{1+(1+t)^b}
      \ = \ \sum_{k=0}^\infty \bigg(\frac{1-(1+t)^b}{2}\bigg)^k\\
      & = & 1 + \sum_{k=1}^\infty
            \frac{(-1)^k}{2^k}\big((1+t)^b-1\big)^k
      \ = \ 1 + \sum_{k=1}^\infty \frac{(-1)^k}{2^k}
            \bigg(\sum_{i=1}^\infty \binom{b}{i}t^i\bigg)^k\\
      & = & 1 + \sum_{k=1}^\infty \frac{(-1)^k}{2^k}
            \sum_{i_1,\ldots,i_k\in\nz}
            \binom{b}{i_1}\cdots\binom{b}{i_k} t^{i_1+\cdots+i_k}\\
      & = & 1 + \sum_{k=1}^\infty \frac{(-1)^k}{2^k}
            \sum_{n=k}^\infty t^n \sum_{\substack{i_1,\ldots,i_k\in\nz\\ i_1+\cdots+i_k=n}} \binom{b}{i_1}\cdots\binom{b}{i_k}\\
      & = & 1 + \sum_{n=1}^\infty t^n
            \sum_{k=1}^n \frac{(-1)^k}{2^k}
            \sum_{\substack{i_1,\ldots,i_k\in\nz\\i_1+\cdots+i_k=n}}.
            \binom{b}{i_1}\cdots\binom{b}{i_k}.
   \end{eqnarray*}
   On the other hand,
   $S_b(t)
   =\int(1+tx)^{-b}\mu({\rm d}x)
   =\int\sum_{n=0}^\infty \binom{-b}{n}(tx)^n\,\mu({\rm d}x)
   =\sum_{n=0}^\infty t^n\binom{-b}{n}\mu_n$.
   Comparing the coefficient in front of $t^n$ yields the solution
   \[
   \mu_n\ =\ \frac{1}{\binom{-b}{n}}
   \sum_{k=1}^n \frac{(-1)^k}{2 ^k}
   \sum_{\substack{i_1,\ldots,i_k\in\nz\\i_1+\cdots+i_k=n}}
   \binom{b}{i_1}\cdots\binom{b}{i_k},
   \qquad n\in\nz.
   \]
   The last multi sum is expressed by a single sum (see, for example,
   \cite[Eq.~(13)]{Moehle2010} as
   \begin{equation} \label{multisum}
      \sum_{\substack{i_1,\ldots,i_k\in\nz\\i_1+\cdots+i_k=n}}
      \binom{b}{i_1}\cdots\binom{b}{i_k}
      \ =\ \sum_{\ell=1}^k (-1)^{k-\ell}\binom{k}{\ell}\binom{b\ell}{n},
      \qquad n,k\in\nz,
   \end{equation}
   which yields the first formula in (\ref{mom}). Using well-known formulas for the standard Stirling numbers $s(.,.)$ and $S(.,.)$ of the first and second kind respectively we obtain
   \begin{eqnarray*}
      &   & \hspace{-14mm}
      \sum_{\ell=1}^k (-1)^{k-\ell}\binom{k}{\ell}\binom{b\ell}{n}
      \ = \ \frac{1}{n!}\sum_{\ell=1}^k
            (-1)^{k-\ell}\binom{k}{\ell}(b\ell)_n
      \ = \ \frac{1}{n!}\sum_{\ell=1}^k(-1)^{k-\ell}\binom{k}{\ell}
            \sum_{j=1}^n s(n,j)(b\ell)^j\\
      & = & \frac{1}{n!}\sum_{j=1}^n s(n,j)b^j
            \sum_{\ell=1}^k(-1)^{k-\ell}\binom{k}{\ell}\ell^j
      \ = \ \frac{k!}{n!}\sum_{j=1}^n s(n,j)b^jS(j,k),
   \end{eqnarray*}
   which yields the second formula in (\ref{mom}).
\end{proof}
\begin{acknowledgement}
The authors would like to thank Frank Redig for many enlightening discussions, in particular for his suggestion of the dual process when the measure $\Lambda$ is the uniform distribution.
The authors also thank the University of T\"ubingen and the Modena and Reggio Emilia University for supporting two one-month research visits where parts of this work were conducted.
The whole probability group at the University of T\"ubingen and the mathematical statistical physics group at the Modena and Reggio Emilia University are warmly acknowledged.

\end{acknowledgement}


\end{document}